\documentclass[10pt]{amsart}

\usepackage{filecontents}

\usepackage{amsmath}
\usepackage{amssymb}
\usepackage{amsthm}
\usepackage{amsfonts, dsfont}
\usepackage{paralist}
\usepackage{graphics} 
\usepackage{epsfig} 
\usepackage{graphicx}  
\usepackage{epstopdf}
\usepackage{epstopdf}
\usepackage{verbatim}
\usepackage{mathrsfs}
\usepackage{mathtools}
\usepackage{pstricks}
\usepackage{cleveref}
\usepackage{relsize}
\usepackage{tikz}
\usetikzlibrary{matrix}
\usepackage{subcaption}
\usepackage{pgfplots}
\usepackage{fixltx2e}
\usepackage{enumitem}
\usepackage{ upgreek }
\usepackage{bm}
\usepackage{appendix}
\usepackage{xcolor}

\usepackage[abbrev]{amsrefs}

\usepackage{amssymb}

\usepackage[all,cmtip]{xy}

\newcommand{\R}{\mathbb{R}}

\usepackage{tikz}
\usepackage{tikz-cd}

\newcommand{\ov}{\overline}
\newcommand{\ep}{\varepsilon}
\newcommand{\la}{\lambda} 

\numberwithin{equation}{section}

\theoremstyle{plain} 
\newtheorem{thm}{Theorem}[section]

\newtheorem{lem}{Lemma}[section]
\newtheorem{prop}{Proposition}[section]

\theoremstyle{defn}

\theoremstyle{remark}
\newtheorem{rem}{Remark}[section]

\definecolor{ForestGreen}{RGB}{34,139,34}
\definecolor{ao(english)}{rgb}{0.0, 0.5, 0.0}

\begin{document}

\title[Singular perturbations with unbounded data]{Singular perturbations in stochastic optimal control with unbounded data}
\thanks{The first author is member of the Gruppo Nazionale per l'Analisi Matematica, la Probabilit\`a e le loro Applicazioni (GNAMPA) of the Istituto Nazionale di Alta Matematica (INdAM). He also  participates in the King Abdullah University of Science and Technology (KAUST) project CRG2021-4674 ``Mean-Field Games: models, theory, and computational aspects''. 
\\
\indent
The 
 second author is funded by the Deutsche Forschungsgemeinschaft (DFG, German Research Foundation) – Projektnummer 320021702/GRK2326 – Energy, Entropy, and Dissipative Dynamics (EDDy). The results of this paper are part of his Ph.D. thesis \cite{kouhkouh22phd} which was conducted when he was a Ph.D. student at the University of Padova.
}

\author{Martino Bardi}
\address {Martino Bardi \newline \indent
{Department of Mathematics “T. Levi-Civita”, \newline \indent 
University of Padova, via Trieste, 63},
\newline \indent
{I-35121 Padova, Italy}
}
\email{\texttt{bardi@math.unipd.it}},

\author{Hicham Kouhkouh}
\address{Hicham Kouhkouh \newline \indent
{RWTH Aachen University, Institut f\"ur Mathematik,  \newline \indent 
RTG Energy, Entropy, and Dissipative Dynamics,\newline \indent
Templergraben 55 (111810)},
 \newline \indent 
{52062, Aachen, Germany}
}
\email{\texttt{kouhkouh@eddy.rwth-aachen.de}}





\date{\today}

\begin{abstract}
We study singular perturbations of a class of two-scale stochastic control systems with unbounded data. The assumptions are designed to cover some relaxation problems for deep neural networks. We construct effective Hamiltonian and initial data and prove the convergence of the value function to the solution of a limit (effective) Cauchy problem for a parabolic equation of HJB type. We use methods of probability, viscosity solutions and homogenization.
\end{abstract}


\maketitle






\section{Introduction}

In this paper we study the asymptotic behavior as $\varepsilon\to 0$ of a system of controlled two-scale
 stochastic differential equations 
\begin{equation}\label{eq: SDE eps intro}
\tag{$SDE\left(\frac{1}{\varepsilon}\right)$}
\begin{aligned}
    \text{d}X_{t} &= f(X_{t},Y_{t},u_{t})\,\text{d}t + \sqrt{2}\,\sigma^{\varepsilon}(X_{t},Y_{t},u_{t})\,\text{d}W_{t},\\
    \text{d}Y_{t} & = \frac{1}{\varepsilon}\,b(X_{t},Y_{t})\,\text{d}t + \sqrt{\frac{2}{\varepsilon}\,}\varrho(X_{t},Y_{t})\,\text{d}W_{t},
\end{aligned}
\end{equation}
where $X_{t}\in \mathds{R}^n$ is the \textit{slow} dynamics, $Y_{t}\in \mathds{R}^m$ is the \textit{fast} dynamics, $u_{t}$ is the control taking values in a given compact set $U$ and $W_{t}$ is a multidimensional Brownian motion.  
We will allow the components of the drift and the diffusion of the slow dynamics to be unbounded and with at most linear growth in the fast variables $Y$. While the diffusion coefficient of the process $X$ can be degenerate (i.e. $\sigma^{\varepsilon}=0$ is allowed), the diffusion coefficient of the process $Y$ is required to be nondegenerate, in particular we will assume for our main result that $\varrho\varrho^{\top}$ is the identity matrix times a positive constant, in addition to other structural assumptions on the data that we shall make precise  later. We carry our analysis in the context of stochastic optimal control problems with payoff functional
\begin{equation*}
    J(t,x,y,u):=\mathds{E}\left[ 
    e^{\lambda(t-T)}g(X_{T},Y_{T}) + 
    \!\int_{t}^{T}\!\ell(s,X_{s},Y_{s},u_{s})e^{\lambda(s-T)}\text{d}s\, \bigg| X_{t}=x,Y_{t}=y
    \right],
\end{equation*}
and exploit that the value function $V^{\varepsilon}(t,x,y):=\sup
_{u}\, J(t,x,y,u)$ solves in the viscosity sense a fully nonlinear parabolic degenerate Hamilton-Jacobi-Bellman PDE in $(0,T)\times\mathds{R}^n\times \mathds{R}^m$.
     
Our motivation comes from the Stochastic Gradient Descent algorithm in the context of Deep Learning and Big Data analysis. The following special case of \eqref{eq: SDE eps intro}, without control, was proposed  in  \cite{chaudhari2018deep} to justify  an algorithm of  Stochastic Gradient Descent called Deep Relaxation (see also \cite{pavon2022local}). Given a {\em loss function} $\phi : \R^n\to\R$ to be minimized, consider its quadratic perturbation in double variables
\begin{equation*}
    \Phi(y,x) := \phi(y) + \frac{1}{2\gamma}|x-y|^{2}
\end{equation*}
and the partial stochastic gradient descent  associated to it
\begin{equation}\label{eq: SDE motiv}
    \begin{aligned}
    \text{d}X_{s} & = -\nabla_{x}\Phi (Y_{s},X_{s})\,\text{d}s,\quad X_{0}=x\in\mathds{R}^{n}\\ 
    \text{d}Y_{s} & = -\frac{1}{\varepsilon}\nabla_{y} \Phi (Y_{s},X_{s})\,\text{d}s + \sqrt{\frac{2}{\varepsilon}}\beta^{-1/2}\,\text{d}W_{s}, \quad Y_{0}=y\in\mathds{R}^{n}.
    \end{aligned}
\end{equation}
The calculations in  \cite{chaudhari2018deep} show that 
one should expect the limit as $\varepsilon\to 0$ in the above system of SDEs to be
\begin{equation*}
    \text{d} \hat{X}_{s} = \int_{\mathds{R}^{n}} -\frac{1}{\gamma}(X_{s}-y)\rho^{\infty}_{_{\beta}}(\text{d}y;X_{s})\,\text{d}s,\quad \hat{X}_{0}=x\in\mathds{R}^{n} ,
\end{equation*}
where $\rho^{\infty}_{_{\beta}}(y;x)$ is the invariant (Gibbs) measure associated to the process $Y_{\cdot}$ ({ with $\varepsilon=1$ and frozen $X_s=x$}). The latter can be written as
\begin{equation*}
    \text{d} \hat{X}_{s} = -\nabla \phi_{\gamma}(\hat{X}_{s})\text{d}s,\quad \hat{X}_{0}=x\in\mathds{R}^{n} ,
\end{equation*}
that is the deterministic gradient descent of the regularized loss function
\begin{equation*}
    \phi_{\gamma}(x) \coloneqq -\frac{1}{\beta}\log\left(G_{\beta^{-1}\gamma}\ast\exp(-\beta \phi(x))\right)
\end{equation*}
where
\begin{equation*}
    G_{\beta^{-1}\gamma}(x) \coloneqq (2\pi\gamma)^{-n/2}\exp\left(-\frac{\beta}{2\gamma}\,|x|^{2}\right)
\end{equation*}
is the heat kernel, and $\beta,\gamma>0$ are parameters used to tune the algorithm. The function $\phi_{\gamma}$ above is called {\em local entropy}, and it is useful in the search of robust minima, because it measures both the depth and the flatness of the valleys in the landscape of the graph of $\phi$, see \cite{chaudhari2019entropy}. Note also that in \eqref{eq: SDE motiv} the drifts are $f=(y-x)/\gamma$ and $b=-\nabla\phi+(x-y)/\gamma$, which are unbounded in $x$ and $y$.

In the present paper, under rather general assumptions, we prove the convergence  { as $\varepsilon\to0$ of the value functions $V^\ep$ associated with the  the  singularly perturbed control system \eqref{eq: SDE eps intro} to a function $V(t,x)$ independent of $y$, and $V$ is characterized as the unique viscosity solution of a Cauchy problem for a limiting HJB equation. The  {\em effective Hamiltonian} $\bar H$ driving such equation and  the  {\em effective initial data} $\bar g$ are explicitly computed by suitable averages. In particular, the result applies to the model problem \eqref{eq: SDE motiv} if $\phi\in C^{1}(\mathds{R}^{n})$ is bounded from below and such that $\nabla\phi$ is Lipschitz continuous with constant $L$, and $\gamma$ is small enough ($\gamma < \frac{1}{L}$). This holds also if the equation for $X_{\cdot}$ in \eqref{eq: SDE motiv} involves a control $u_s$, e.g., it is of the form
\begin{equation}\label{modified}
\text{d}X_{s}  = - u_s \nabla_{x}\Phi(Y_{s},X_{s})
\end{equation}
where $u_s\in [0,1]$ is the {\em learning rate} of the SGD algorithm. This variant is used in the companion paper \cite{bardi2022deep} to prove that by \eqref{eq: SDE motiv} modified with \eqref{modified} one reaches in expectation a value of $\phi$ lower than the one got by classical stochastic gradient descent. In \cite{bardi2022deep} we also characterize explicitly the limiting system of controlled SDEs in $\R^n$ and prove results on the convergence of the trajectories of  \eqref{eq: SDE eps intro} to the trajectories of such effective system as $\varepsilon\to0$.} These results can be found also in the second author's thesis \cite{kouhkouh22phd}.

Our convergence theorem includes the previous results in \cite{bardi2011optimal,bardi2010convergence},  where the coefficients in the slow variable were assumed to be bounded with respect to the fast variables and mostly viscosity method for the HJB PDE were employed. However, some important parts of the proofs in \cite{bardi2011optimal,bardi2010convergence} do not work in the current setting. Here we use first a truncation to big balls of the cell problem, an HJB equation of  ergodic type that formally gives the effective Hamiltonian, and then use probabilistic estimates on the exit time $\tau^Y_n$ of the process
\begin{equation}
    \label{fast s}
    \text{d}Y_{t}= b(x,Y_{t})\,\text{d}t + \sqrt{2}\varrho(x,Y_{t})\,\text{d}W_{t} 
\end{equation}
from balls of radius $n$ as $n\to \infty$. This approach is new to our knowledge in the present context. Here some ideas are borrowed from \cite{herrmann2006transition}.

Our results also allow to generalise several applications of singular perturbations to finance, e.g., models of pricing and trading derivative securities in financial markets with stochastic volatility, as in \cite{bardi2010convergence, fouque2011multiscale}, applications in economics and advertising theory, as in \cite{bardi2011optimal}, and connections to large deviations as in \cite{feng2012small, spiliopoulos2013large, ghilli2018viscosity}.

There is a wide literature on singular perturbations for ODEs and control systems that goes back to the late 60's, see \cite{kokotovic1999singular} and the references therein, and for diffusion processes, with and without control, see \cite{kushner1990WeakCM, bardi2011optimal,bardi2010convergence} and their bibliographies. We mention also the series of papers \cite{pardoux2001poisson,pardoux2003poisson,pardoux2005poisson} by Pardoux and Veretennikov on the approximation of diffusions without control from the point of view of Poisson equation, and the contributions by  Borkar and Gaitsgory  \cite{borkar2007averaging, borkar2007singular} on singularly perturbed stochastic differential equations with control both in the slow and in the fast variables, relying on the {Limit Occupational Measure Set}.  
More recent results for uncontrolled SDEs were obtained in \cite{liu2020averaging, rockner2021averaging} under weaker regularity assumptions and in \cite{de2021averaging} for nonautonomus systems with an application in finance.
LQ problems with multiplicative noise were treated in \cite{goldys2020multiscale}.  
Some extensions to infinite dimensional control systems were studied in \cite{swikech2021singular} and \cite{guatteri2021singular}.
 Other results were obtained using different techniques from nonlinear filtering theory in \cite{athreya2021simultaneous}. The very recent paper \cite{ghilli2022rate} studies the rate of convergence in an unbounded setting.

The paper is organized as follows. In \textbf{Section \ref{sec: setting}} we present the two scale stochastic control problem and the assumptions that will hold throughout the paper, together with the associated Hamilton-Jacobi-Bellman equation. \textbf{Section \ref{sec: ergodic}} is devoted to the study of ergodicity properties of the process \eqref{fast s}, the estimates on $\tau^Y_n$, and the construction of the effective Hamiltonian and initial data and of suitable approximate correctors for the singularly perturbed HJB equation.
This is a crucial  step for the convergence result of the value function that we next show in \textbf{Section \ref{sec: conv val}}.  In this last section  we rely on viscosity methods, with  an adaptation of the Evans' perturbed test function method \cite{evans1989perturbed} to fit our unbounded context. 

\section{The two scale stochastic control problem}
\label{sec: setting}

\subsection{The stochastic system}
\label{sec:stochastic system}
Let $(\Omega,\mathcal{F},\mathcal{F}_{t},\mathds{P})$ be a complete filtered probability space and let $(W_{t})_{t}$ be an $\mathcal{F}_{t}$-adapted standard $r$-dimensional Brownian motion. We consider the following stochastic control system
\begin{equation}
\label{dynamics}
\left\{\;
\begin{aligned}
    \text{d}X_{t} &= f(X_{t},Y_{t},u_{t})\,\text{d}t + \sqrt{2}\,\sigma^{\varepsilon}(X_{t},Y_{t},u_{t})\,\text{d}W_{t},\quad X_{0} = x \in\mathds{R}^{n}\\
    \text{d}Y_{t} & = \frac{1}{\varepsilon}\,b(X_{t},Y_{t})\,\text{d}t + \sqrt{\frac{2}{\varepsilon}}\,\varrho(X_{t},Y_{t})\,\text{d}W_{t},\quad Y_{0}=y\in\mathds{R}^{m}
\end{aligned}
\right.
\end{equation}
Throughout the paper, we shall make different sets of assumptions that we will present when needed. We start with the following

 \textit{Assumptions \textbf{(A)}}
\begin{enumerate}[label=\textbf{(A\arabic*)}]
    \item For a given compact set $U$, $f:\mathds{R}^{n}\times\mathds{R}^{m}\times U \rightarrow \mathds{R}^{n}$, 
$\sigma^{\varepsilon}: \mathds{R}^{n}\times\mathds{R}^{m}\times U \rightarrow \mathds{M}^{n,r}$, 
$b:\mathds{R}^{n}\times\mathds{R}^{m}\rightarrow \mathds{R}^{m}$ and $\varrho:\mathds{R}^{n}\times\mathds{R}^{m}\rightarrow \mathds{M}^{m,r}$
 are continuous functions, Lipschitz continuous in $(x,y)$ uniformly with respect to $u\in U$ and $\varepsilon>0$, and with linear growth in both $x$ and $y$, that is,
\begin{equation}
\label{assumption-slow}
    |f(x,y,u)|,\|\sigma^{\varepsilon}(x,y,u)\|\leq C(1+|x|+|y|),\quad \forall\; x,y,\;\forall \varepsilon>0
\end{equation}
\begin{equation}
\label{assumption-fast}
    |b(x,y|,\|\varrho(x,y)\|\leq C(1+|x|+|y|),\quad \forall\; x,y,
\end{equation}
for some positive constant $C$.
\item The diffusion $\sigma^{\varepsilon}$ driving the slow variables $X_{t}$ satisfies
\begin{equation*}
    \lim\limits_{\varepsilon\rightarrow0}\sigma^{\varepsilon}(x,y,u) = \sigma(x,y,u)\quad\text{locally uniformly},
\end{equation*}
where $\sigma:\mathds{R}^{n}\times\mathds{R}^{m}\times U \rightarrow \mathds{M}^{n,r}$ satisfies the same conditions as $\sigma^{\varepsilon}$. We will not make any nondegeneracy assumption on the matrices $\sigma^{\varepsilon},\sigma$, so the cases $\sigma^{\varepsilon},\sigma\equiv0$ are allowed.
\item The diffusion $\varrho$ driving the fast variables $Y_t$ is such that $\varrho\varrho^{\top}$ is uniformly bounded and non degenerate, i.e. $\exists\; \underline{\Lambda}, \;\overline{\Lambda}\;>0,$ such that $\forall\;x,y,\xi.$
\begin{equation}
\label{non degen}
    \underline{\Lambda} |\xi|^{2}\leq \varrho(x,y)\varrho^{\top}(x,y)\xi\cdot\xi = |\varrho(x,y)^{\top}\xi|^{2}\leq \overline{\Lambda}|\xi|^{2}.
\end{equation}
\item The following \textit{recurrence condition} holds for the fast variables $Y_{\cdot}$
\begin{equation}
\label{recurrence condition}
    \forall\;x\in \mathds{R}^n, \exists\;A_{x},R_{x}>0 \; \text{ s.t.}\quad b(x,y)\cdot y < - A_{x} |y|,\quad \forall\;|y|\geq R_{x}.
\end{equation}
\end{enumerate}
We will simply denote $A=A_{x},R=R_{x}$ when there is no confusion. Note that condition \eqref{recurrence condition} is  related to the one introduced by Pardoux and Veretennikov in \cite{pardoux2001poisson}  namely, $\lim_{|y|\rightarrow\infty} \sup_{x\in\mathds{R}^{n}} b(x,y)\cdot y = -\infty$ uniformly in $x$, and usually called Khasminskii's assumption. It will be strengthened later into assumption (C2) to get better properties of the invariant measure and effective Hamiltonian.

\subsection{The optimal control problem}
\label{sec:OCP sys}
We define the following pay off functional for a finite horizon optimal control problem associated to system \eqref{dynamics} for $t\in[0,T]$
\begin{equation}
\label{cost function}
    J(t,x,y,u):=
    \mathds{E}\left[ 
    e^{\lambda(t-T)}g(X_{T},Y_{T}) + 
    \int_{t}^{T}\ell(s,X_{s},Y_{s},u_{s})e^{\lambda(s-T)}\text{d}s \; \bigg|\; X_{t}=x,\;Y_{t}=y
    \right].
\end{equation}
The associated value function is
\begin{equation}
    \label{value function}
    \tag{$\;OCP(\varepsilon)\;$}
    V^{\varepsilon}(t,x,y) := \sup\limits_{u\in\mathcal{U}}J(t,x,y,u),\quad \text{subject to }\; \eqref{dynamics}\,.
\end{equation}
The set of admissible control functions $\mathcal{U}$ is the standard one in stochastic control problems, i.e., 
the set of $\mathcal{F}_{t}$-progressively measurable processes taking values in $U$. We will make the following

\textit{Assumptions \textbf{(B)}}
\begin{enumerate}[label=\textbf{(B\arabic*)}]
    \item The discount factor is $\lambda\geq0$.
    \item The utility function $g:\mathds{R}^{n}\times\mathds{R}^{m}\rightarrow\mathds{R}$ and the running cost $\ell:[0,T]\times\mathds{R}^{n}\times\mathds{R}^{m}\times U\rightarrow\mathds{R}$ are continuous functions and satisfy
\begin{equation}
\label{assumption-cost}
    \exists\;K>0\;\text{ s.t. }\; |g(x,y)|,|\ell(s,x,y,u)|\leq K(1+|x|^{2} + |y|^{2}),\;\forall s\in[0,T], x,y.
\end{equation}
\item {The running cost $\ell$ is locally H\"older continuous in $y$ uniformly in $u$, i.e., for any $R>0$ and $s, x$ there are constants $\gamma, C>0$ such that
\begin{equation}
\label{assumption-cost2}
|\ell(s,x,y,u)-\ell(s,x,\tilde y,u)|\leq C |y-\tilde y|^\gamma \quad \forall \, |y|\leq R, u\in U.
\end{equation}  }
\end{enumerate}

\subsection{The HJB equation}
The HJB equation associated via Dynamic Programming to the value function $V^{\varepsilon}$ is
\begin{equation}
    \label{HJB}
    -V^{\varepsilon}_{t} + F^{\varepsilon}\left(t,x,y,V^{\varepsilon},D_{x}V^{\varepsilon},\frac{D_{y}V^{\varepsilon}}{\varepsilon}, D^{2}_{xx}V^{\varepsilon},\frac{D^{2}_{yy}V^{\varepsilon}}{\varepsilon}, \frac{D^{2}_{x,y}V^{\varepsilon}}{\sqrt{\varepsilon}}\right) = 0 , \; 
    \text{ in }\;(0,T)\times\mathds{R}^{n}\times\mathds{R}^{m},
\end{equation}
complemented with the obvious terminal condition
\begin{equation}
\label{terminal cond}
    V^{\varepsilon}(T,x,y)=g(x,y) .
\end{equation}
This is a fully nonlinear degenerate parabolic equation (strictly parabolic in the $y$ variables by the assumption \eqref{non degen}). We denote by $\mathds{M}^{n,m}$ (respec. $\mathds{S}^{n}$) the set of matrices of $n$ rows and $m$ columns (respec. the subset of $n$-dimensional squared symmetric matrices). The Hamiltonian $F^{\varepsilon}:[0,T]\times\mathds{R}^{n}\times\mathds{R}^{m}\times\mathds{R}\times\mathds{R}^{n}\times\mathds{R}^{m}\times\mathds{S}^{n}\times\mathds{S}^{m}\times\mathds{M}^{n,m}\rightarrow\mathds{R}$ is defined as
\begin{equation}
    \label{Ham F}
    F^{\varepsilon}(t,x,y,r,p,q,M,N,Z) := H^{\varepsilon}(t,x,y,p,M,Z) - \mathcal{L}(x,y,q,N)+\lambda r,
\end{equation}
where
\begin{equation}
    \label{H-eps}
    H^{\varepsilon}(t,x,y,p,M,Z) := \min\limits_{u\in U}\left\{ -\text{trace}(\sigma^{\varepsilon}\sigma^{\varepsilon\top}M) - f\cdot p - 2\text{trace}(\sigma^{\varepsilon}\varrho^{\top}Z^{\top}) - \ell  \right\},
\end{equation}
with $\sigma^{\varepsilon},f$   computed at $(x,y,u)$, $\ell=\ell(t,x,y,u)$, and $\varrho=\varrho(x,y)$, and
\begin{equation}
    \label{inf gen}
    \mathcal{L}(x,y,q,N) := b(x,y)\cdot q + \text{trace}(\varrho(x,y)\varrho^{\top}(x,y)N)
\end{equation}
We define also the Hamiltonian $H$ as $H^{\varepsilon}$ with $\sigma^{\varepsilon}$ is replaced by $\sigma$
\begin{equation}
    \label{H}
    H(t,x,y,p,M,Z) := \min\limits_{u\in U}\left\{ -\text{trace}(\sigma\sigma^{\top}M) - f\cdot p - 2\text{trace}(\sigma\varrho^{\top}Z^{\top}) - \ell  \right\}.
\end{equation}

The next result is standard, see, e.g., \cite[Proposition 3.1]{bardi2010convergence} or \cite[Proposition 2.1]{bardi2011optimal}. 
\begin{prop}
\label{prelimit sol}
 {Under assumptions (A) and (B1,B2)}, for any $\varepsilon>0$, the function $V^{\varepsilon}$ in \eqref{value function} is the unique continuous viscosity solution to the Cauchy problem \eqref{HJB}-\eqref{terminal cond} with at most quadratic growth in $x$ and $y$, i.e., $\exists\;K>0$ independent by $\varepsilon$   such that
\begin{equation}
\label{quadratic gro-V-eps}
    |V^{\varepsilon}(t,x,y)|\leq K(1+|x|^{2}+|y|^{2}),\quad \forall\;t\in [0,T],\;x\in\mathds{R}^{n},\;y\in\mathds{R}^{m}.
\end{equation}
\end{prop}
Note that  the functions $V^{\varepsilon}$ are locally equibounded but can be unbounded. The unboundedness in $y$ was not allowed in the previous literature on singular perturbations and it is the main difficulty and novelty in this paper.

\section{Ergodicity of the fast variables and { the effective limit problem}}
\label{sec: ergodic}

\subsection{The invariant measure}
\label{inv_meas}

Consider the diffusion processes in $\mathds{R}^{m}$ obtained by setting $\varepsilon=1$ in \eqref{dynamics} and freezing $x\in\mathds{R}^{n}$
\begin{equation}
    \label{fast subsys}
    \text{d}Y_{t}= b(x,Y_{t})\,\text{d}t + \sqrt{2}\varrho(x,Y_{t})\,\text{d}W_{t},\quad Y_{0}=y\in\mathds{R}^m
\end{equation}
called \textit{fast subsystem}. If we want to recall the dependence on the parameter $x$, we denote the process in \eqref{fast subsys} as $Y^{x}_{\cdot}$. Observe that its infinitesimal generator is $\mathcal{L}_{x}w := \mathcal{L}(x,y,D_{y}w,D^{2}_{yy}w)$ with $\mathcal{L}$ defined by \eqref{inf gen}. 
 
Let us recall that a probability measure $\mu_{x}$ on $\R^m$ is an {\it invariant measure} for the process $Y^{x}_{\cdot}$ in \eqref{fast subsys} if 
\begin{equation}
    \int_{\mathds{R}^{m}}\mathds{E}[f(Y_{t})\,|Y_0 = y]\,\text{d}\mu_{x}(y) = \int_{\mathds{R}^m}f(y)\,\text{d}\mu_{x}(y),\quad \forall\,t>0 ,
\end{equation}
for all bounded Borel functions $f$ in $\mathds{R}^m$ (see for example \cite{lorenzi2006analytical}). We recall that an invariant measure is a  stationary solution of the Fokker-Planck equation $\mathcal{L}_{x}^{*}\mu_{x} = 0$, where $\mathcal{L}_{x}^*$ is the adjoint operator to $\mathcal{L}_{x}$. When it exists and is unique we say that the process $Y^{x}_{\cdot}$ is ergodic.

It is well known that the assumption  \eqref{recurrence condition} on the drift ensures the existence of an invariant measure for \eqref{fast subsys}, and its uniqueness follows from the non-degeneracy assumption \eqref{non degen} on the diffusion $\varrho$. This is proven for instance in \cite{veretennikov1997polynomial} (see also \cite{pardoux2001poisson,pardoux2003poisson,pardoux2005poisson}). Another proof of existence and uniqueness of the invariant measure is in \cite{bardi2010convergence} assuming the existence of a  Lyapunov-type function, which is related to the recurrence condition \cite{bardi2011optimal}.

In  this section we drop the explicit dependence on the frozen $x$. Instead, we stress the dependence of $Y_{\cdot}$ on its  initial position $y$ by writing 
\begin{equation}
    \label{fast subsys - no x}
    \text{d}Y_{y}(t)=b(Y_{y}(t))\,\text{d}t + \sqrt{2}\varrho(Y_{y}(t))\,\text{d}W_{t},\quad Y_{y}(0)=y\in\mathds{R}^{m} .
\end{equation}

\subsection{Auxiliary results}
\label{aux}

The first result we need is the following lemma which  gives a stronger form of ergodicity of the fast subsystem, that is, the convergence of the probability law of $Y_{y}(\cdot)$ towards its unique invariant probability measure. We use $\|\mu-\nu\|_{TV}$ for the total variation distance between two probability measures $\mu,\nu$ defined by
\begin{equation*}
    \|\mu-\nu\|_{TV}= \sup\limits_{A\in\mathcal{B}}|\mu(A) - \nu(A)|
\end{equation*}
where $\mathcal{B}$ is the class of Borel sets. In particular, $\|\mu\|_{TV} = \int_{\mathds{R}^{m}}\text{d}\mu=1$. 
\begin{lem}
\label{conv-prob law}
Under  assumptions (A), there exists $C, d,k>0$ such that
\begin{equation}
\label{tv}
    \|\mathds{P}_{Y_{y}(t)}(\cdot)  -  \mu(\cdot)\|_{TV} \leq C(1+|y|^{d})(1+t)^{-(1+k)} .
\end{equation}
 Moreover, the invariant measure $\mu$ has finite moments of any order.
\end{lem}

\begin{proof}
This is a particular case of the more general result in \cite[Theorem 6]{veretennikov1997polynomial}. Indeed, the main assumption in \cite{veretennikov1997polynomial} is
\begin{equation}
\label{assumption-ver}
    \exists\;M_{0}\geq 0,\; r\geq 0 \;\text{ s.t. }\quad b(y)\cdot y  \leq -r,\quad \forall\;|y|\geq M_{0} .
\end{equation}
Then, for the constants
\begin{equation*}
         \Tilde{\Lambda} := \sup_{y}
         {\text{trace}(\varrho\varrho^{\top}(y))}/{m}, \quad
         r_{0} := [r - (m\Tilde{\Lambda} - \underline{\Lambda})/2]\overline{\Lambda}^{-1} ,
\end{equation*}   
Theorem 6 in \cite{veretennikov1997polynomial} states that \eqref{tv} holds $\forall \; k\in (0,r_{0}-\frac{3}{2})$, $\forall\; d\in (2k+2, 2r_{0}-1)$ if  $r_{0}>\frac{3}{2}$. In our case, assumption \eqref{recurrence condition} guarantees a constant $r$, and therefore $r_{0}$,  as large as we want.

For the finite moments, see \cite[eq. (28) in \S 6]{veretennikov1997polynomial}, where it is shown that the invariant measure has finite moments of order $d\in(2k+2, 2r_{0}-1)$  if $k\in (0,r_{0}-\frac{3}{2})$. It is enough to use H\"older inequality together with the fact that $\mu(\mathds{R}^{m})=1$ to prove finite moments of any order $d\geq 1$.
\end{proof}

The following result gives an estimate on  the first exit time of $Y_{y}(\cdot)$ from the ball centered in $0$ with radius $n$
$$
\tau_{n}^{Y}:=\inf\left\{t \geq 0\;|\; \|Y_{y}(t)\|\geq n\right\} .
$$ 
It will be needed together with the previous Lemma for constructing the limit PDE in the next section. 
\begin{lem}
\label{conv-exit time}
Under assumptions (A), for any compact set $\mathcal K$, there exist $\eta,C$ positive constants and $\ell$ a positive function such that, for any $\delta\in (0,1)$ and for $n$ large enough,
\begin{equation}
\label{conv-exit time - 1}
    \mathds{E}\left[ e^{-\delta \tau_{n}^{Y}} \right] \leq C \frac{\ell(\delta)}{\delta}e^{-n\eta}, \quad \forall \, y\in \mathcal K ,
\end{equation} 
where $\ell(\delta)= 1 + O(\delta)$ when $\delta\to 0^{+}$. In particular for any $\alpha\geq 0$ and $\beta>0$, one has
\begin{equation}
    \mathds{E}\left[n^{\alpha}e^{-\frac{1}{n^{\beta}}\tau_{n}^{Y}}\right] \leq C n^{\alpha+\beta}e^{-n\eta}\,\longrightarrow 0 \;\text{as }\; n\to+\infty.
\end{equation}
\end{lem}

\begin{proof}
The idea of the proof is to build  a process $Z_{t}\in\mathds{R}$ such that $\|Y(t)\|\leq Z_{t}$ a.s.. Then one has $\tau_{n}^{Y}\geq \tau_{n}^{Z}$ a.s., where $\tau_{n}^{Z}:=\inf\left\{t \geq 0\;|\; |Z(t)|\geq n\right\}$, and hence 
\begin{equation}\label{comparison}
  \mathds{E}\left[e^{-\delta\tau_{n}^{Y}}\right] \leq \mathds{E}\left[e^{-\delta\tau_{n}^{Z}}\right],\quad \forall\;\delta>0.
\end{equation}
Once we will have such a process $Z$, we'll give an upper bound of the right hand side  in \eqref{comparison}.

The construction of $Z$ such that $\|Y_{t}\|\leq Z_{t}$ a.s. is inspired by the proof of \cite[Proposition 1.4]{herrmann2006transition}. Let $A$ and $R$ be the positive constants in the recurrence condition  \eqref{recurrence condition} and note that $R$ can be chosen as large as we want. Define $h:\mathds{R}^{m}\rightarrow \mathds{R}$ as a $C^{2}$ function such that $h(y) = \|y\|$ when $\|y\|\geq R$, and $h(y) < R$ otherwise. Next define
\begin{equation}
    Z_{t} := R\vee \|y_{o}\| +\sqrt{2} M_{t} - \eta \xi_{t} + L_{t} ,
\end{equation}
where $Y_{0}=y_{o}$, $\eta$ is a positive constant to be made precise,
$$
M_{t}:=\int_{0}^{t}\nabla h(Y_{s})^{\top}\varrho(Y_{s})dW_{s}, 
\quad t\geq 0,
$$
$$
\xi_{t}:=\int_{0}^{t}\|\nabla h(Y_{s})^{\top}\varrho(Y_{s})\|^{2}ds\;
$$
 is the quadratic variation of the continuous local martingale $M_{t}$, and $L_{t}$ is an increasing process (of finite variation) which increases only at times $t$ for which $Z_{t}=R$, and is of zero value when $Z>R$.
Such pair $(Z,L)$ is the unique pair of continuous adapted process given by Skorokhod's lemma (see e.g. \cite[chap.VI, \S 2]{revuz2013continuous}): $Z$ is a process reflected  out of the interval $]-R, R[$ and $L$ its compensator. 
Note that when $\|y\|\geq R$, $\nabla h(y) = \frac{y}{\|y\|}$ so
$\left\|\nabla h(y)^{\top}\varrho(y)\right\|^{2} = \frac{1}{\|y\|^{2}} y^{\top}\varrho(y)\varrho(y)^{\top}y\leq\overline{\Lambda}$
by \eqref{non degen}, and hence $d\xi_{t} \leq \overline{\Lambda}\, dt$ on $\{\|Y_{t}\|\geq R\}$. On the other hand, define $\Tilde{K}:=\sup\limits_{\|y\|\leq R}\|\nabla h(y)\|^{2}$. Then we have $\xi_{t}\leq (1\vee \Tilde{K})\,\overline{\Lambda}\, t$ for all $t\geq 0$. We set $K:= (1\vee \Tilde{K})\overline{\Lambda}$, and get
\begin{equation}
    \label{K for xi}
    0\leq\xi_t\leq Kt,\quad\forall\; t\geq 0 .
\end{equation}
Now we choose $f\in C^{2}(\mathds{R})$ such that
\begin{equation*}
    \begin{aligned}
    &f(x)>0 \quad \text{and} \quad f'(x) >0, & \forall \; x>0\\
    &f(x)=0, & \forall \; x\leq 0
    \end{aligned}
\end{equation*}
We set $a(y):=\varrho(y)\varrho(y)^{\top}$. According to It\^o's formula, for $t\geq 0$,
\begin{equation*}
    \begin{aligned}
        \text{d}h(Y_{s}) & = \left(\nabla h(Y_{s})^{\top}b(Y_{s})  + \text{trace}\left(a(Y_{s})D^{2} h(Y_{s})\right)\right)\text{d}s + \sqrt{2}\nabla h(Y_{s})^{\top}\varrho(Y_{s})\text{d}W_{s} \\
         \text{d}Z_{s} & = -\eta d\xi_{s} + \text{d}L_{s} + \sqrt{2}\text{d}M_{s}\\
        & = -\eta\left\|\nabla h(Y_{s})^{\top}\varrho(Y_{s})\right\|^{2}\text{d}s + \text{d}L_{s} + \sqrt{2}\nabla h(Y_{s})^{\top}\varrho(Y_{s})\text{d}W_{s}
    \end{aligned}
\end{equation*}
so that
\begin{equation*}
    \text{d}\left(h(Y)-Z\right)_{s} = \left(\nabla h(Y_{s})^{\top}b(Y_{s})  + \text{trace}\left(a(Y_{s})D^{2} h(Y_{s})\right) + \eta\left\|\nabla h(Y_{s})^{\top}\varrho(Y_{s})\right\|^{2} \right)\text{d}s  - \text{d}L_{s} .
\end{equation*}
Again by It\^o's formula we obtain 
\begin{equation*}
    \begin{aligned}
    f(h(Y_{t})-Z_{t}) =& f(h(y_{o}) - R\vee\|y_{o}\|) + \int_{0}^{t}f'(h(Y_{s})-Z_{s})\text{d}(h(Y)-Z)_{s}+\\
    & \quad \quad \quad + \frac{1}{2}\int_{0}^{t}f''(h(Y_{s})-Z_{s})\text{d}\langle h(Y) - Z \rangle_{s},
    \end{aligned}
\end{equation*}
where $\langle\zeta\rangle_{t}=\int_{0}^{t}\sigma(\zeta_{s})\sigma^{\top}(\zeta_{s})\text{d}s$ denotes the quadratic variation of a process defined by $\text{d}\zeta_{t} = f(\zeta_{t})\text{d}t + \sigma(\zeta_{t})\text{d}W_{t}$.

Note that  $ f(h(y_{o}) - R\vee\|y_{o}\|)=0$ by definition of $h$ and $f$. Moreover, $h(Y_{\cdot})-Z_{\cdot}$ is a continuous process with no Wiener process term, and hence it has zero quadratic variation, i.e., $\text{d}\langle h(Y) - Z \rangle_{s}=0$.
Now, again by definition of $h$ and $Z$, we have $h(Y_{t})\leq Z_{t}$ on $\{\|Y_{t}\|\leq R\}$, so $\{h(Y_{t})>Z_{t}\}=\{\|Y_{t}\|>Z_{t}\}$ is a subset of $\{\|Y_{t}\|> R\}$. 
When $\|y\|\geq R$, we have $\nabla h(y) = \frac{y}{\|y\|}$, $D^{2}h(y) = \frac{1}{\|y\|}\left(\mathds{I}_{m} - \frac{y\otimes y}{\|y\|^{2}}\right)$, and we compute, by \eqref{non degen},
\begin{equation}
    \text{trace}\left(a(y)D^{2}h(y)\right)  = \frac{1}{\|y\|}\left(\text{trace} \, a(y) - \sum\limits_{i,j=1}^{m}a_{ij}(y)\frac{y_{i}y_{j}}{\|y\|^{2}}\right)      \leq \frac{m\overline{\Lambda} }{\|y\|}.
\end{equation}
Hence the expression
\begin{equation*}
    \begin{aligned}
    \int_{0}^{t}f'(\|Y_{s}\|-Z_{s})\left\{ \frac{1}{\|Y_{s}\|}Y_{s}\cdot b(Y_{s}) + \frac{m}{\|Y_{s}\|}\overline{\Lambda} + \eta\overline{\Lambda}\, \right\}\text{d}s -\int_{0}^{t}f'(\|Y_{s}\|-Z_{s})\text{d}L_{s},
    \end{aligned}
\end{equation*}
which is valid  for $\|Y_{s}\|> R$, is an upper bound of $f(h(Y_{t})-Z_{t})$. Furthermore, $\text{d}L_{s}=0$ for $\|Y_{s}\|> R$, and therefore one has
\begin{equation}
\label{est_above}
    f(h(Y_{t})-Z_{t}) \leq \int_{0}^{t}f'(\|Y_{s}\|-Z_{s})\left\{ \frac{1}{\|Y_{s}\|} Y_{s}\cdot b(Y_{s}) + \frac{m }{\|Y_{s}\|}\overline{\Lambda} + \eta\overline{\Lambda}\, \right\}\text{d}s
\end{equation}
By  the recurrence condition \eqref{recurrence condition} the quantity in brackets $\left\{...\right\}$ is bounded from above by
\begin{equation*}
    \begin{aligned}
    -A + \frac{m}{\|Y_{s}\|}\overline{\Lambda} + \eta\overline{\Lambda} & \leq -A + \frac{m}{R}\overline{\Lambda} + \eta\overline{\Lambda}\, .
    \end{aligned}
\end{equation*}
because  this upper bound is obtained for $\|Y_{s}\|> R$. Now we choose $R$ large enough so that  $A/\overline{\Lambda}> m/R$. Then for 
\begin{equation*}
   0< \eta < \frac{A}{\overline{\Lambda}}-\frac{m}{R}
\end{equation*}
the r.h.s. of \eqref{est_above}  is negative,
which ensures
 $f(h(Y_{t})-Z_{t})\leq 0$ and implies $\|Y_{t}\|\leq Z_{t}$ a.s. by definition of $f$. 

Next we look for an upper bound to $\mathds{E}\left[e^{-\delta \tau_{n}^{Z}}\right]$. For simplicity of notation, in this step we shall write $\tau_{n}:= \tau_{n}^{Z}$, dropping the dependence on $Z$. Fix $\delta\in (0,1)$ and  set $\gamma := \frac{\delta}{K}$ where $K$ is the constant in \eqref{K for xi}. By It\^o's formula, for any $\Phi \in C^{2}(\mathds{R})$,
\begin{equation*}
\begin{aligned}
    \text{d}(\Phi(Z_{t})e^{-\gamma \xi_{t}}) 
    = &\;  \sqrt{2}\Phi'(Z_{t})e^{-\gamma \xi_{t}}\text{d}M_{t} + \Phi'(Z_{t})e^{-\gamma \xi_{t}}\text{d}L_{t}\\
    & \quad \quad \quad \quad  + e^{-\gamma \xi_{t}}\left\{ \Phi''(Z_{t}) - \eta\Phi'(Z_{t}) - \gamma \Phi(Z_{t}) \right\}\text{d}\xi_{t}.
\end{aligned}
\end{equation*}
Since we are interested in the limit as $n\to \infty$, we can assume without loss of generality that $n> R$. We choose $\Phi$ such that
\begin{equation}
\label{differential equ}
\left\{\;
    \begin{aligned}
        & \Phi''(z) - \eta \Phi'(z) - \gamma \Phi(z)=0, & \text{for }\; z\in [R,n]\\
        & \Phi'(R)=0 \quad \text{ and } \quad  \Phi(n) = 1& 
    \end{aligned}
\right.
\end{equation}
then $\Phi(Z_{t})e^{-\gamma \xi_{t}}$ is a local martingale which is bounded up to time $\tau_{n}$. Hence we are allowed to apply Doob's stopping theorem to obtain
\begin{equation}
    \Phi(R\vee \|y_{o}\|) = \mathds{E}\left[\Phi(Z_{\tau_{n}})e^{-\gamma \xi_{\tau_{n}}}\right]
\end{equation}
and since $Z_{\tau_{n}}=n$, $\Phi(n)=1$, and $\xi_{t}\leq Kt$ for all $t\geq 0$, we have
\begin{equation}
\label{doob-K}
     \mathds{E}\left[e^{-\gamma K\tau_{n}}\right] \leq \mathds{E}\left[e^{-\gamma \xi_{\tau_{n}}}\right] = \Phi(R\vee \|y_{o}\|)
\end{equation}
which yields
\begin{equation}
\label{doob}
     \mathds{E}\left[e^{-\delta \tau_{n}}\right] \leq \Phi(R\vee \|y_{o}\|)
\end{equation}
Now solving the differential equation \eqref{differential equ} yields
\begin{equation}
\label{sol-H}
    \Phi(z) = 
    \frac{-\lambda_{2}e^{\lambda_{1}(z-R)} + \lambda_{1}e^{\lambda_{2}(z-R)}}{-\lambda_{2}e^{\lambda_{1}(n-R)} + \lambda_{1}e^{\lambda_{2}(n-R)}} ,
\end{equation}
where $\lambda_{2}<0<\lambda_{1}$ are given by
\begin{equation*}
    \lambda_{1} = \frac{1}{2}\left(\eta + \sqrt{\eta^{2}+4\gamma}\right),\quad \quad \lambda_{2} = \frac{1}{2}\left(\eta - \sqrt{\eta^{2} + 4\gamma}\right) .
\end{equation*}
Hence,
\begin{equation*}
    \begin{aligned}
    \Phi(z) & \leq \frac{(\lambda_{1} - \lambda_{2})e^{\lambda_{1}(z-R)}}{-\lambda_{2}e^{\lambda_{1}(n-R)}}\\
        & \leq 2 \frac{\sqrt{1+\frac{4\gamma}{\eta^{2}}}}{\sqrt{1+\frac{4\gamma}{\eta^{2}}} - 1}\text{exp}\left[z \,  \frac{\eta}{2}\left(1+\sqrt{1+\frac{4\gamma}{\eta^{2}}}\right)\right]\text{exp}\left[-n\frac{\eta}{2}\left(1+\sqrt{1+\frac{4\gamma}{\eta^{2}}}\right)\right] 
    \end{aligned}
\end{equation*}
By Taylor expansion, for $\gamma$ small, 
\begin{equation*}
	\begin{aligned}
	1 + 2\frac{ \gamma}{\eta^{2}} - 2 \frac{\gamma^{2}}{\eta^{4}} \leq \sqrt{1+\frac{4\gamma}{\eta^{2}}}  \leq 1 + 2\frac{ \gamma}{\eta^{2}} ,
	\end{aligned}
\end{equation*}
which yields 
\begin{equation*}
	\Phi(z) \leq \frac{1+2\frac{\gamma}{\eta^{2}}}{\frac{\gamma}{\eta^{2}}-\frac{\gamma^{2}}{\eta^{4}}  }\text{exp}\left[z \, \eta\left(1+\frac{\gamma}{\eta^{2}}\right)\right]e^{-n\frac{\eta}{2}}.
\end{equation*}
Now recall that $\gamma := \frac{\delta}{K}$ and define
\begin{equation*}
    \ell(\delta) := \frac{1+\frac{2\delta}{K\eta^{2}}}{1-\frac{\delta}{K\eta^{2}}  }\text{exp}\left[z\frac{\delta}{K\eta}\right] \quad \text{and }\; C:=\eta^{2}Ke^{z\eta} .
\end{equation*}
Then the right-hand side in the last inequality equals $C\frac{\ell(\delta)}{\delta}e^{-n\frac{\eta}{2}}$ and  $\ell(\delta) = 1+O(\delta)$ when $\delta\to 0^{+}$. Together with \eqref{doob}, for $z:=R\vee\|y_{o}\|$ this  yields
\begin{equation*}
    \mathds{E}\left[e^{-\delta {\tau_{n}^Z}}\right] \leq C\frac{\ell(\delta)}{\delta}e^{-n\frac{\eta}{2}} .
\end{equation*}
By combining this  inequality with  \eqref{comparison}   we finally get the desired estimate \eqref{conv-exit time - 1} and  conclude the proof of the first statement. 

The second statement of the lemma is immediately obtained by multiplying the inequality \eqref{conv-exit time - 1} by $n^{\alpha}$ for $\alpha\geq0$ and choosing $\delta=n^{-\beta}$ for $\beta>0$. 
\end{proof}

\begin{rem}
\label{lem: exit time - quant}
 {Under  assumptions (A),} we can also prove that, for suitable $C_{1},C_{2}$ and $\kappa>0$,  
\begin{equation}
    C_{2}\,(n^{2}-|y|^{2})\,\leq \mathds{E}[\,\tau_{n}^Y\,]\,\leq \, C_{1}\,e^{\kappa\,n^{2}} \quad \text{locally uniformly in } y .
\end{equation}
\end{rem}

\subsection{The effective Hamiltonian and approximate correctors}\label{seq: eff ham}

We expect that the effective Hamiltonian in the limit HJB equation of the singular perturbation problem is
\begin{equation}
    \label{eff-Hamiltonian}
    \overline{H}({t},{x},{p},{P}) :=  \int_{\mathds{R}^{m}} H({t},{x},y,{p},{P},0)d\mu_{x}(y)
\end{equation}
where $\mu_x$ is the invariant measure of the process \eqref{fast subsys} introduced in Section \ref{inv_meas} and studied in Section \ref{aux}. In classical periodic homogenization theory one proves the convergence by means of a corrector (see \cite{lions1987homogenization, evans1989perturbed, alvarez2002viscosity}), namely, a (periodic) solution  $\chi$  {(for fixed $(t,x,p,P)$)} of the  \textit{cell problem}
\begin{equation*}
    -\mathcal{L}({x},y,D\chi,D^{2}\chi) + H({t},{x},y,{p},{P},0) = \ov H \quad \text{in }\; \mathds{R}^{m},
\end{equation*}
where $\mathcal{L}$ and $H$ are defined in \eqref{inf gen} and \eqref{H}. In many cases, however, the cell problem may be hard or impossible to solve, and then one resorts to   \textit{approximate correctors}, i.e., a sequence $\chi_n$ such that 
\begin{equation*}
    -\mathcal{L}({x},y,D\chi_n,D^{2}\chi_n) + H({t},{x},y,{p},{P},0) \to  \ov H \quad \text{locally uniformly, }
\end{equation*}
see, e.g., \cite{alvarez2003singular, bardi2011optimal}. Here  the unboundedness of both the domain and the Hamiltonian does not allow us to build the approximate correctors globally. We overcome the problem by introducing a suitable
\textit{truncated} $\delta$\textit{-cell problem} that we now describe.
Fix $(t, {x},{p},{P})$, and let us denote for simplicity 
\begin{equation}
   \mathcal{L}\omega(y)  := \mathcal{L}(x,y,D\omega,D^{2}\omega) ,
\end{equation}
\begin{equation}
    h(y) := H({t},{x},y,{p},{P},0)\quad \text{in }\mathds{R}^m .
\end{equation}
Note that $h$ is locally H\"older continuous by the assumptions (A) and (B).

Take a sequence of bounded and open domains $D_{n}$ such that $ \ov D_{n}\subset D_{n+1}$  and $\cup_nD_{n}=  \mathds{R}^{m}$.  Assume in addition that $\partial D_n$ is  $C^{2}$ and  $D_{n}\subseteq B(0,n):=\{y\in\mathds{R}^{m}\;|\; \|y\|< n\}$, the open ball centered in $0$ with radius $n$ (e.g., $D_n=B(0,n)$). Consider  the Dirichlet-Poisson problem
\begin{equation}
\label{dirichlet-Dn}
\left\{
    \begin{aligned}
        \;\delta u(y) -\mathcal{L}u(y) &= -h(y) ,\; \text{in } \; D_{n} ,&\\
        u(y) & = 0,\; \text{on } \partial D_{n} .&
    \end{aligned}
\right.
\end{equation}
It has a unique solution $u^{\delta,n}(\cdot)$ (see, e.g., \cite[Theorem 8.1, p.79]{mao2007stochastic}) given by
\begin{equation}
\label{sol dirichlet sto}
    u^{\delta,n}(y) =  \mathds{E}\left[-\int_{0}^{\tau_{n}} h(Y_{y}(t))e^{-\delta t} \text{d}t\right]
\end{equation}
where $\tau_{n}$ is the first exist time of $Y_{y}(\cdot)$ from $D_{n}$. 
In the next result we study  the limit as $\delta\to 0$ and $n\to \infty$.  We will use that, under the assumptions \eqref{assumption-slow} and \eqref{assumption-cost}, the Hamiltonian has at most a quadratic growth in $y$, i.e.
\begin{equation}
\label{growth h}
    \exists\; K_{h}>0,\quad |h(y)|\leq K_{h}(1+|y|^{2}),\quad \forall\;y\in\mathds{R}^{m}
\end{equation}
where $K_{h}$ is a constant that depends on the slow dynamics data $(f,\sigma)$ and the running cost $\ell$.

\begin{prop}
\label{eff ham}
Let $u^{\delta,n}(\cdot)$ be the solution to \eqref{dirichlet-Dn}. Under assumptions (A) and (B), for any  $\alpha>0$  and $\delta=\delta(n) = O\left(\frac{1}{n^{4+\alpha}}\right)$ as $n\to\infty$, 
\begin{equation}
\label{boundary-phi}
    \begin{aligned}
     \lim\limits_{n\rightarrow\infty}\left|\delta(n) u^{\delta(n),n}(y) + \mu(h)\right|  =0,\quad \text{ locally uniformly in } y,
    \end{aligned}
\end{equation}
where $\mu(h) = \int_{\mathds{R}^{m}}h(y)d\mu(y)= \overline{H}({t},{x},{p},{P})$ and $\mu$ is the unique invariant probability measure for the process \eqref{fast subsys}.
\end{prop}

\begin{proof}
From \eqref{sol dirichlet sto}, for $D_{n}^{c}:=\mathds{R}^{m}\setminus D_{n}$, we have 
\begin{equation*}
    \begin{aligned}
        & u^{\delta,n}(y) + \frac{\mu(h)}{\delta} \\& = \mathds{E}\left[-\int_{0}^{\tau_{n}} h(Y_{y}(t))e^{-\delta t} \text{d}t\right] + \int_{0}^{\infty}\int_{\mathds{R}^{m}}h(y)e^{-\delta t}d\mu(y)\text{d}t 
        \\
        &=\mathds{E}\left[-\int_{0}^{\infty}\mathds{1}_{D_{n}}(Y_{y}(t))h(Y_{y}(t))e^{-\delta t} \text{d}t\right] + \int_{0}^{\infty}\int_{\mathds{R}^{m}}h(y)e^{-\delta t}d\mu(y)\text{d}t \\
        & \quad \quad \quad \quad \quad \quad \quad \quad + \mathds{E}\left[\int_{\tau_{n}}^{\infty}\mathds{1}_{D_{n}}(Y_{y}(t))h(Y_{y}(t))e^{-\delta t} \text{d}t\right] 
        \\
        & = \int_{0}^{\infty}\int_{D_{n}}h(y)d\bigg(\mu(y)-\mathds{P}_{Y_{y}(t)}(y)\bigg)e^{-\delta t}\text{d}t + \frac{1}{\delta}\int_{D_{n}^{c}}h(y)d\mu(y)\\
        & \quad \quad \quad \quad \quad \quad \quad \quad + \mathds{E}\left[\int_{\tau_{n}}^{\infty}\mathds{1}_{D_{n}}(Y_{y}(t))h(Y_{y}(t))e^{-\delta t} \text{d}t   \right] .
    \end{aligned}
\end{equation*}
To estimate the first term we apply first H\"older inequality to get 
\begin{equation*}
    \begin{aligned}
        &\left|\int_{0}^{\infty}\int_{D_{n}}h(y)\text{d}\bigg(\mu(y)-\mathds{P}_{Y_{y}(t)}(y)\bigg)e^{-\delta t}\text{d}t \right|\\
        & \quad \quad \quad \quad \leq \left(\int_{0}^{\infty}\bigg(\int_{D_{n}}h(y)\text{d}(\mu(y)-\mathds{P}_{Y_{y}(t)}(y))\bigg)^{2}\text{d}t\right)^{1/2}
        \left(\int_{0}^{\infty}e^{-2\delta t}\text{d}t\right)^{1/2} \\
        &  
        \quad \quad \quad \quad = \frac{1}{\sqrt{2\delta}}\left(\int_{0}^{\infty}\left(\int_{D_{n}}h(y)\text{d}(\mu(y)-\mathds{P}_{Y_{y}(t)}(y))\right)^{2}\text{d}t\right)^{1/2} 
    \end{aligned}
\end{equation*}
Now we can bound the term in the r.h.s. by Lemma \ref{conv-prob law} and \eqref{growth h} as follows
\begin{equation*}
    \begin{aligned}
        \int_{0}^{\infty}\left(\int_{D_{n}}h(y)\text{d}(\mu(y)-\mathds{P}_{Y_{y}(t)}(y))\right)^{2}\text{d}t &  \leq
\int_{0}^{\infty}\left(\sup\limits_{D_n}|h|\int_{D_{n}}\text{d}(\mu(y)-\mathds{P}_{Y_{y}(t)}(y))\right)^{2}\text{d}t \\
        & \leq \sup\limits_{D_n}|h|^{2}\int_{0}^{\infty}\big\|\mathds{P}_{Y_{y}(t)}(\cdot)-\mu(\cdot)\big\|^{2}_{TV}\text{d}t\\
        & \leq \frac{C^2(1+ |y|^{d})^2}{1+2k}\sup\limits_{D_n}|h|^{2}\\
        & \leq \frac{C^2(1+ |y|^{d})^2}{1+2k}K_{h}^{2}(1+n^{2})^{2}.
    \end{aligned}
\end{equation*}
Finally, we have the following upper bound
\begin{equation}
\label{first term}
    \begin{aligned}
        \left|\int_{0}^{\infty}\int_{D_{n}}h(y)\text{d}\!\left(\mu(y)-\mathds{P}_{Y_{y}(t)}(y)\right)e^{-\delta t}\text{d}t \right|\leq K_{h}  \frac{C(1+ |y|^{d})}{1+2k} \frac{(1+n^{2})}{\sqrt{2\delta}} .
    \end{aligned}
\end{equation}
We rewrite the second term as
\begin{equation}
\label{second term}
    \frac{1}{\delta}\int_{D_{n}^{c}}h(y)\text{d}\mu(y) = \frac{1}{\delta}\left(\mu(h) - \int_{D_{n}}h(y)\text{d}\mu(y)\right)
\end{equation}
We bound the third term using the definition of $D_{n}$ and \eqref{growth h} 
\begin{equation}
\label{third term}
    \begin{aligned}
        \left|\mathds{E}\left[\int_{\tau_{n}}^{\infty}\mathds{1}_{D_{n}}(Y_{y}(t))h(Y_{y}(t))e^{-\delta t} \text{d}t \right] \right| & \leq 
        K_{h}\mathds{E}\left[\int_{\tau_{n}}^{\infty}\mathds{1}_{D_{n}}(Y_{y}(t))(1+|Y_{y}(t)|^{2})e^{-\delta t} \text{d}t\right]
        \\
        & \leq K_{h}\mathds{E}\left[\int_{\tau_{n}}^{\infty}\mathds(1+n^{2})e^{-\delta t} \text{d}t\right] \\
        & \leq   K_{h}\frac{1+n^{2}}{\delta}
    \mathds{E}\left[e^{-\delta \tau_{n}}\right]
        \end{aligned}
\end{equation}
Now we add up \eqref{first term}, \eqref{second term}, and \eqref{third term}, and multiply by $\delta$ to get 
\begin{equation*}
    \begin{aligned}
         \left|\delta u^{\delta,n}(y) + \mu(h)\right| & \leq
          K_{h}\sqrt{\delta} \,  \frac{C(1+ |y|^{d})(1+n^{2})}{(1+2k)\sqrt{2}}  +  \left|\mu(h)-\int_{D_{n}}h(y)\text{d}\mu(y)\right| \\
          & \quad \quad \quad \quad \quad \quad \quad \quad +   K_{h}(1+n^{2}) 
        \mathds{E}\left[e^{-\delta \tau_{n}}\right] .
    \end{aligned}
\end{equation*}
If  we set $\delta=\delta(n) = O(\frac{1}{n^{4+\alpha}})$ with $\alpha>0$, the last term converges to zero as $n\to\infty$ by Lemma \ref{conv-exit time}, and then
$$
\lim\limits_{n\rightarrow\infty}\left|\delta(n) u^{\delta(n),n}(y) + \mu(h)\right|  =0 .
$$
\end{proof}

\begin{rem}
This result still holds true if we relax the growth condition \eqref{growth h} on  $h$ to 
$$
\exists\; K_{h}>0,\quad |h(y)|\leq K_{h}(1+|y|^{\gamma}),\quad \forall\;y\in\mathds{R}^{m} ,
$$
with  any $\gamma\geq 0$, provided we set $\delta=O\left(\frac{1}{n^{2\gamma + \alpha}}\right)$ in Proposition \ref{eff ham}. This means that the slow dynamics is allowed to have a polynomial growth w.r.t. the fast variables. The same result holds also if $u^{\delta,n}(\cdot)$ satisfies a inhomogeneous boundary condition $u(y)  = \phi(y)$ on $\partial D_{n}$ in the Dirichlet problem \eqref{dirichlet-Dn}, if  $\phi$ has a polynomial growth, that is, $\exists\;K_{\phi}>0$ and  $\kappa\geq 0$ such that  $|\phi(y)|\leq K_{\phi}(1+|y|^{\kappa})$.  The proof requires only minor modifications, see \cite{kouhkouh22phd}.
\end{rem}

The next result is an \textit{exchange property} which allows the effective Hamiltonian $\overline{H}$ to be of Bellman type. Such representation will be useful for applying a comparison theorem in the conclusion of our main result.
\begin{prop}
\label{representation-prop}
Under assumptions (A) and (B2), the effective Hamiltonian \eqref{eff-Hamiltonian} can be written as 
\begin{equation}
    \label{representation}
    \overline{H}(t,x,p,P) = \min\limits_{\nu\in L^{\infty}(\mathds{R}^{m},U)}\int_{\mathds{R}^{m}}\; \left[     -\text{trace}(\sigma\sigma^{\top}P) - f\cdot p - \ell \right]\;\text{d}\mu_{x}(y)
\end{equation}
where $\sigma,f$ are computed in { $(x,y,\nu(y))$} and $\ell$ in { $(t,x,y,\nu(y))$}.
\end{prop}
Note that $L^{\infty}(\mathds{R}^{m},U)=L^{1}((\mathds{R}^{m},\mu_{x}), U)$ because $U$ is bounded and $\mu_{x}$ is a finite measure.
\begin{proof}
Let $t,x,p,P$ be fixed and define
\begin{equation*}
    F(y,u) := -\text{trace}(\sigma(x,y,u)\sigma(x,y,u)^{\top}P) - f(x,y,u)\cdot p - \ell(t,x,y,u),
\end{equation*}
so that ${H}(t,x,y,p,P,0)=\min_{u \in U}F(y,u)$.
To prove the inequality ``$\,\leq\,$", it suffices to observe that for any $\varepsilon>0$, there exists $\nu^{\varepsilon}\in L^{\infty}(\mathds{R}^{m},U)$ such that
\begin{equation}
\label{ineq}
\begin{aligned}
    \inf\limits_{\nu\in L^{\infty}(\mathds{R}^{m},U)}\int_{\mathds{R}^{m}}F(y,\nu(y))\,\text{d}\mu_{x}(y) + \varepsilon\;  & \geq \int_{\mathds{R}^{m}}F(y,\nu^{\varepsilon}(y))\,\text{d}\mu_{x}(y)\\
&  \geq \int_{\mathds{R}^{m}}\,\min\limits_{u\in U}F(x,u)\,\text{d}\mu_{x}(y) =     \overline{H} 
\end{aligned}
\end{equation}
and hence the result by the arbitrariness of $\varepsilon$.

To prove the inequality ``$\geq$", we consider the minimization problem 
\begin{equation*}
\mathfrak{F}(y) := \min\limits_{u \in U}F(y,u)
\end{equation*}
 where $y\in \mathds{R}^{m}$. Since $F$ is  continuous, $U$ is compact, $\mathfrak{F}(y)\in F(\{y\}\times U)$, and $\mathfrak{F}$ is continuous, a classical selection theorem (see \cite[Theorem 7.1, p. 66]{Himmelberg1975}) implies the existence of a measurable selector $\overline{\nu}$ for which the minimization is achieved, i.e.,
\begin{equation*}
    \exists\;\overline{\nu}\in L^{\infty}(\mathds{R}^{m},U),\;\text{s.t. } \;\forall \; y \in \mathds{R}^{m},\; \mathfrak{F}(y) = \min\limits_{u\in U}F(y,u) = F(y,\overline{\nu}(y)).
\end{equation*}
Therefore one has
\begin{equation*}
\begin{aligned}
  \overline{H} =  \int_{\mathds{R}^{m}}\min\limits_{u\in 
 U}F(y,u)\,\text{d}\mu_{x}(y)  & = \int_{\mathds{R}^{m}}F(y,\overline{\nu}(y))\,\text{d}\mu_{x}(y))\\
&  \geq \inf\limits_{\nu\in L^{\infty}(\mathds{R}^{m},U)}\int_{\mathds{R}^{m}}F(y,\nu(\cdot))\,\text{d}\mu_{x}(y).
\end{aligned}
\end{equation*}
This inequality together with \eqref{ineq} proves that the inf is a min, attained at $\nu=\overline{\nu}$, and the equality \eqref{representation} holds.
\end{proof}

\subsection{The effective initial data}\label{seq: eff init dat}

In this section we construct the effective terminal cost $\ov g(x)$ for the limit of the singular perturbations problem \eqref{HJB}-\eqref{terminal cond}. 
We expect that it is 
\begin{equation}
    \label{eff-g}
    \overline{g}({x}) :=  \int_{\mathds{R}^{m}} g({x},y)d\mu_{x}(y)
\end{equation}
where $\mu_x$ is the invariant measure of the process \eqref{fast subsys}. In classical  homogenization theory one uses that 
 \[
\overline{g}({x})= \lim\limits_{t\to+\infty}\omega(t,y;{x})
 \]
 where $\omega$ solves, for fixed $x$, the  initial value problem:
\begin{equation}
\label{eq: cauchy prob_initial data}
\left\{
    \begin{aligned}
        \omega_{t} - \mathcal{L}(x, y,D\omega,D^{2}\omega) = 0&\quad \text{in }  (0,+\infty)\times \mathds{R}^m,\\
         \omega(0,y) = g(x,y), &\quad \text{in } \mathds{R}^m,
    \end{aligned}
\right.
\end{equation}
with $\mathcal{L}\,$  defined in \eqref{inf gen}, see, e.g., \cite{alvarez2003singular, bardi2011optimal}. In our context of unbounded data we use a truncation to bounded domains of such a problem, similar to the previous section. We consider an increasing sequence of bounded and open domains $D_{n}$ with $C^{2}$ boundaries invading $\R^m$  and such that $D_{n}\subseteq B(0,n)$, as in \S \ref{seq: eff ham} (for example $D_{n}=B(0,n)$). Now instead of \eqref{eq: cauchy prob_initial data}, we consider the Cauchy-Dirichlet problem 
\begin{equation}
\label{eq: cauchy prob_initial data - truncated}
\left\{
    \begin{aligned}
        \frac{\partial}{\partial t}\omega^{T,n} - \mathcal{L}(x, y,D\omega^{T,n},D^{2}\omega^{T,n}) = 0, &\quad \text{in }  (0,T]\times D_{n},\\
         \omega^{T,n}(0,y) = g(x,y), &\quad \text{in } D_{n},\\
         \omega^{T,n} (t,y) = 0, &\quad \text{in } [0,T]\times \partial D_{n},
    \end{aligned}
\right.
\end{equation}
where $x$ is again a fixed parameter, and if we set $u^{T,n}(t,y) = \omega^{T,n}(T-t,y)$, then $u^{T,n}(\cdot,\cdot)$ solves the terminal-boundary value problem
\begin{equation}
\label{eq: initial-boundary value problem}
\left\{
    \begin{aligned}
         \frac{\partial}{\partial \,t}u^{T,n} + \mathcal{L}(x, y,Du^{T,n},D^{2}u^{T,n}) = 0, & \quad \text{in }  [0,T)\times D_{n},\\
         u^{T,n}(T,y) = g(x,y), &\quad \text{in } D_{n},\\
         u^{t,n} (t,y) = 0, &\quad \text{in } [0,T]\times \partial D_{n},
    \end{aligned}
\right.
\end{equation}
It is known \cite[Theorem 8.2, p.81]{mao2007stochastic} that the problem \eqref{eq: initial-boundary value problem} admits a unique solution given by 
\begin{equation}
\label{eq: sol - cauchy prob - initial data - truncated}
    u^{T,n}(t,y) = \mathds{E}[\,\mathds{1}_{\{\tau_{n}\wedge T=T\}}\,g(x,Y_{y,t}( T))\,]
\end{equation}
where $Y_{y,t}(\cdot)$ is the fast process defined by \eqref{fast subsys} and such that $Y_{y,t}(t) = y \in \mathds{R}^m$, and $\tau_{n} = \inf\{s\in [t,T]\,:\,Y_{y,t}(s)\notin D_{n}\}$ is the first exit time from $D_n$. The next result gives an approximation of the effective initial data $\ov g(x)$ by $\omega^{T,n}(T,y) = u^{T,n}(0,y)$  as  $T=T(n)\to +\infty$ for $n\to +\infty$. 

\begin{prop}
\label{prop: effective initial data}
Let $u^{T,n}(\cdot,\cdot)$ be as defined in \eqref{eq: sol - cauchy prob - initial data - truncated}. Under assumptions (A) and (B), for any increasing sequence $\{T(n)\}_{n>0}$ such that $T(n)\geq \,n^{2}$, we have the following
\begin{equation}
    \lim\limits_{n\to +\infty} \bigg|\,u^{T(n),n}(0,y) \,-\,\overline{g}  \,\bigg| =0,\quad \text{locally uniformly in }y,
\end{equation}
where $\overline{g}= \overline{g} (x) = \int_{\mathds{R}^m}g(x,y)\text{d}\mu_{x}(y)$ and $\mu_{x}$ is the unique invariant probability measure of the process \eqref{fast subsys}. In particular
$\lim\limits_{n\to+\infty}\omega^{T(n),n}(T(n),y) =\overline{g}\;$ locally uniformly in $y$ and $\overline{g}$ has at most quadratic growth in $x$. 
\end{prop}

\begin{proof} 
Since  the slow variable $x$ is frozen we drop it in the notations and write in particular $g(x,\cdot)=g(\cdot)$ and $\mu_{x}(\cdot)=\mu(\cdot)$. Also, the fast process $Y_{y,0}(\cdot)$
 will be simply denoted by $Y_{y}(\cdot)$. We have the following
\begin{equation*}
    \begin{aligned}
        & u^{T(n),n}(0,y) = \int_{D_{n}} \mathds{1}_{\{\tau_{n}\wedge T(n)=T(n)\}} g(z)\,\text{d}\mathds{P}_{Y_{y}(T(n))}(z),\quad \text{from }\, \eqref{eq: sol - cauchy prob - initial data - truncated}\\
        & \overline{g} = 
         \int_{\mathds{R}^m}g(z)\,\text{d}\mu(z) = \int_{D_{n}}g(z)\,\text{d}\mu(z) + \int_{D_{n}^{c}}g(z)\,\text{d}\mu(z).
    \end{aligned}
\end{equation*}
Hence
\begin{equation*}
    \begin{aligned}
       \bigg|\,u^{T(n),n}(0,y) \,-\,\overline{g}  \,\bigg| & \leq \, \left|\int_{D_{n}} \mathds{1}_{\{\tau_{n}\wedge T(n) = T(n) \}} g(z)\,\text{d}\!\left(\mathds{P}_{Y_{y}(T(n))}-\mu\right)(z)\right|  + \left|\int_{D_{n}^{c}}g(z)\,\text{d}\mu(z)\right|\\
       & \leq C(1+n^{2}) \|\mathds{P}_{Y_{y}(T(n))}(\cdot)-\mu(\cdot)\|_{TV} + \sqrt{\mu(g^{2})}\sqrt{1-\mu(D_{n})}
    \end{aligned}
\end{equation*}
where, for the first integral we used  the quadratic growth of $g$ from \eqref{assumption-cost}, and for the second,  H\"older inequality together with the fact that the probability measure $\mu$ has finite  fourth moment  by Lemma \ref{conv-prob law}. Now, again by  Lemma \ref{conv-prob law}, there exist $C,d,k>0$ such that
\begin{equation*}
    \begin{aligned}
        \|\mathds{P}_{Y_{y}(T(n))}(\cdot)-\mu(\cdot)\|_{TV} \leq C(1+|y|^{d})(1+T(n))^{-(1+k)}
    \end{aligned}
\end{equation*}
Therefore, by choosing $T(n)\geq\,n^{2}$ we obtain, as $n\to\infty$, 
\begin{equation*}
    \bigg|\,u^{T(n),n}(0,y) \,-\,\mu(g)  \,\bigg| \leq C(1+n^{2})(1+|y|^{d})\frac{1}{(1+n^{2})^{1+k}} \,+\,  \sqrt{\mu(g^{2})(1-\mu(D_{n}))}\, 
    \rightarrow\,0.
\end{equation*}
Finally, the growth condition on $\overline{g}$ follows from 
\eqref{assumption-cost} and the fact that $\mu$ has a finite second order moment (Lemma \ref{conv-prob law}).
\end{proof}

\begin{rem}
This result still holds true if we consider, instead of the growth assumption \eqref{assumption-cost}, $g$ such that
\begin{equation*}
    \exists\,K_{g}>0,\quad |g(x,y)|\leq K_{g}(1+|x|^{2}+|y|^{\gamma}),\quad \forall\,y\in\mathds{R}^m
\end{equation*}
where $\gamma\geq0$ is as large as we want, provided we choose $T(n)\,\geq\, n^{\gamma}$.
\end{rem}

\section{The convergence theorem for the value function}
\label{sec: conv val}

We can now state and prove the main result of the paper, namely  the  convergence as $\varepsilon\rightarrow0$ of the value function $V^{\varepsilon}(t,x,y)$, solution to \eqref{HJB}-\eqref{terminal cond}, to a function $V(t,x)$  characterised as the unique solution of the Cauchy problem
\begin{equation}
\label{CP-limit HJB}
\left\{\;
\begin{aligned}
    -V_{t} + \overline{H}(t,x,D_{x}V,D^{2}_{xx}V) + \lambda V(x) & = 0, & \text{in }\; (0,T)\times \mathds{R}^{n} ,\\
    \quad V(T,x) & = \overline{g}(x),& \text{in }\; \mathds{R}^{n} ,
\end{aligned}
\right.
\end{equation}
where the effective Hamiltonian $\overline{H}$ and the effective initial data $\overline{g}(x)$ are defined by \eqref{eff-Hamiltonian} and \eqref{eff-g}, respectively.

Before we go further, we need to check smoothness in the $x$ variables of the data in the effective (limit) Cauchy problem. Indeed, the construction of $\overline{H},\overline{g}$ in the previous section involves the invariant measure of the fast process $Y$ which depends on $x$.

\subsection{On the effective Cauchy problem} 
This subsection is devoted to the continuity of $\overline{H},\overline{g}$. Under the assumptions (A) and (B), the proof of this property reduces to proving continuity of the invariant measure $\mu_{x}$ of the process $Y^{x}_{\cdot}$ in \eqref{fast subsys}. To do so, we need the following

\textit{Assumptions \textbf{(C)}}
\begin{enumerate}[label=\textbf{(C\arabic*)}]
    \item The diffusion $\varrho$  is constant such that $\varrho\varrho^{\top}=\Bar{\varrho\,}\mathds{I}_{m}$ where $\Bar{\varrho}>0$ is a constant and $\mathds{I}_{m}$ is the identity matrix.
    \item The drift $b$ satisfies the following \textit{strong recurrence condition}
    \begin{equation}
        \exists\,\kappa>0 \text{ s.t. } \,(b(x,y_{1})-b(x,y_{2}))\cdot(y_{1}-y_{2}) \leq -\kappa \, |y_{1}-y_{2}|^{2},\quad \forall\, x,y_{1},y_{2}.
    \end{equation}
    \item The utility function $g$ and running cost $\ell$ are Lipschitz continuous in $y$ uniformly in their other arguments.
\end{enumerate}
It is clear that (C3) implies (B3), 
(C2) implies (A4), while (C1) is a particular case of (A3).\\
We recall the weighted norm (when it exists) $\|\varphi\|^{p}_{L^{p}(\nu)}=\int_{\mathds{R}^{m}}|\varphi(y)|^{p}\text{d}\nu(y)$ for $p\geq 1$, $\nu$ a positive measure, and the Wasserstein distance
\begin{equation}
    \mathcal{W}_{p}(\mu,\nu) = \left(\inf\limits_{\pi}\, \iint_{\mathds{R}^{m}}|y-y'|^{p}\,\text{d}\pi(y,y') \right)^{1/p},\quad \text{for } p\geq 1
\end{equation}
where the minimization is  performed over the collection of all measures $\pi$ on $\mathds{R}^{m}\times\mathds{R}^{m}$ having marginals $\mu,\nu$. We now state a result in \cite{bogachev2014kantorovich}.
\begin{lem}\label{lem: bogachev}
Under assumptions (A) and (C) 
\begin{equation*}
    \mathcal{W}_{2}(\mu_{x_{1}},\mu_{x_{2}}) \leq \overline{\varrho}\,\kappa^{-1}\,\|b(x_{1},\cdot) - b(x_{2},\cdot)\|_{L^{2}(\mu_{x_{2}})}
\end{equation*}
where $\mu_{x_{i}}$ is the unique invariant probability measure associated to \eqref{fast subsys}  with $x=x_{i}, i=1,2$, respectively.
\end{lem}

\begin{proof}
The inequality with $\overline{\varrho}=1$ is \cite[Corollary 2]{bogachev2014kantorovich} where it is assumed that $b(x_{1},\cdot),b(x_{2},\cdot)$ satisfy (C2) and such that $|b(x_{1},\cdot)-b(x_{2},\cdot)|\in L^{2}(\mu_{x_{2}}+\mu_{x_{2}})$. This last condition is satisfied as a consequence of Lemma \ref{conv-prob law} which guarantees existence of all moments of the invariant measure in our setting and hence the desired integrability conditions.
\end{proof}

\begin{prop}
\label{prop: cont}
Under assumptions (A), (B) and (C), the effective Hamiltonian $\overline{H} : [0,T]\times \R^n\times\R^n\times\mathbb{S}^n\to\R$ and initial data $\overline{g}:  \R^n\to\R$ are continuous. 
\end{prop}

\begin{proof}
We write the proof for $\overline{H}$ only, $\overline{g}$ being completely analogous. Recall the definition of the effective Hamiltonian $\overline{H}$
\begin{equation*}
    \overline{H}(t, x,p,P) =  \int_{\mathds{R}^{m}} H(t,
    x,y,p,P,0)\text{d}\mu_{x}(y), 
\end{equation*}
where  $\mu_{x}$ is the unique invariant probability measure associated to the fast subsystem \eqref{fast subsys}. The Hamiltonian $H$ inherits all the regularity properties of $f,\sigma,\ell$ as easily seen from its definition \eqref{H}. 
Let $(t_{1},x_{1},p_{1},P_{1}), (t_{2},x_{2},p_{2},P_{2})\in [0,T]\times \mathds{R}^{n}\times \mathds{R}^{n}\times \mathds{S}^{n}$, 
\begin{equation}\label{eq: cont 1}
    \begin{aligned}
        \overline{H}(t_{1},x_{1},p_{1},P_{1}) - \overline{H}(t_{2},x_{2},p_{2},P_{2}) & = \overline{H}(t_{1},x_{1},p_{1},P_{1}) - \overline{H}(t_{1},x_{2},p_{1},P_{1}) +\\
        & \quad \quad \quad \quad  \overline{H}(t_{1},x_{2},p_{1},P_{1})- \overline{H}(t_{2},x_{2},p_{2},P_{2}).
    \end{aligned}
\end{equation}
On one hand we have
\begin{equation}\label{eq: cont 2}
    \begin{aligned}
        &\overline{H}(t_{1},x_{2},p_{1},P_{1})- \overline{H}(t_{2},x_{2},p_{2},P_{2})  \\
        & \quad \quad \quad \quad \quad \quad \quad \quad = \int_{\mathds{R}^{m}} H(t_{1},x_{2},y,p_{1},P_{1},0) - H(t_{2},x_{2},y,p_{2},P_{2},0)\; \text{d}\mu_{x_{2}}(y).
    \end{aligned}
\end{equation}
The variable $x_{2}$ being fixed here, and from continuity of $H$ in $(t,p,P)$, one easily deduces continuity of $\overline{H}(\cdot,x_{2},\cdot,\cdot)$. 
On the other hand, we have
\begin{equation*}
    \begin{aligned}
        & \overline{H}(t_{1},x_{1},p_{1},P_{1}) - \overline{H}(t_{1},x_{2},p_{1},P_{1}) \\
        & \quad \quad \quad \quad = \int_{\mathds{R}^{m}} H(t_{1},x_{1},y,p_{1},P_{1},0)\text{d}\mu_{x_{1}}(y) - \int_{\mathds{R}^{m}} H(t_{1},x_{2},y,p_{1},P_{1},0)\text{d}\mu_{x_{2}}(y)
    \end{aligned}
\end{equation*}
The variables $(t_{1},p_{1},P_{1})$ being fixed, we introduce $\phi(x,y):=H(t_{1},x,y,p_{1},P_{1},0)$. We need then to estimate the quantity
\begin{equation}
\label{eq: cont 3}
    \begin{aligned}
       &  \int_{\mathds{R}^{m}}\phi(x_{1},y)\text{d}\mu_{x_{1}}(y) - \int_{\mathds{R}^{m}}\phi(x_{2},y)\text{d}\mu_{x_{2}}(y)\\
        &\quad \quad \quad \quad \quad \quad  = \int_{\mathds{R}^{m}} \big( \phi(x_{1},y) - \phi(x_{2},y) \big) \text{d}\mu_{x_{1}}(y) + \int_{\mathds{R}^{m}}\phi(x_{2},y)\text{d}(\mu_{x_{1}} -\mu_{x_{2}})(y).
    \end{aligned}
\end{equation}
The first term in the r.h.s. is continuous thanks to continuity of $\phi$ in $x$. We are then left with the second term
\begin{equation}\label{eq: cont 4}
\begin{aligned}
    \int_{\mathds{R}^{m}}\phi(x_{2},y)\text{d}(\mu_{x_{1}}-\mu_{x_{2}})(y) & = \iint_{\mathds{R}^{m}} \phi(x_{2},y)-\phi(x_{2},y')\;\text{d}\pi(y,y')\\
    & \leq C\,\iint_{\mathds{R}^{m}} |y-y'|\;\text{d}\pi(y,y')
\end{aligned}
\end{equation}
for any $\pi(\cdot,\cdot)$ a probability measure on $\mathds{R}^{m}\times\mathds{R}^{m}$ with marginals $\mu_{x_{1}}$ and $\mu_{x_{2}}$. Therefore, we have
\begin{equation}\label{eq: cont 5}
    \iint_{\mathds{R}^{m}} \phi(x_{2},y)-\phi(x_{2},y')\;\text{d}\pi(y,y') \leq C\, \mathcal{W}_{1}(\mu_{x_{1}},\mu_{x_{2}})\leq C\,\mathcal{W}_{2}(\mu_{x_{1}},\mu_{x_{2}})
\end{equation}
Using Lemma \ref{lem: bogachev}, we have the following
\begin{equation*}
    \mathcal{W}_{2}(\mu_{x_{1}},\mu_{x_{2}}) \leq\, \overline{\varrho}\,\kappa^{-1}\,\left(\int_{\mathds{R}^{m}}|b(x_{1},y)-b(x_{2},y)|^{2}\,\text{d}\mu_{x_{2}}(y)\right)^{1/2}
\end{equation*}
and hence
\begin{equation}
\label{eq: cont 6}
    \mathcal{W}_{2}(\mu_{x_{1}},\mu_{x_{2}}) \leq\, \overline{\varrho}\,\kappa^{-1}\,C\,|x_{1}-x_{2}|
\end{equation}
where $C>0$ is now the Lipschitz constant of $b$. Finally, using \eqref{eq: cont 4}, \eqref{eq: cont 5} and \eqref{eq: cont 6} we can upperbound the r.h.s. of \eqref{eq: cont 3} with  
\begin{equation}
\label{eq: cont 7}
       \int_{\mathds{R}^{m}} \big(\phi(x_{1},y) - \phi(x_{2},y)\big) \text{d}\mu_{x_{1}}(y) + \overline{\varrho}\,\kappa^{-1}\,C\,|x_{1}-x_{2}|.
\end{equation}
Finally, exchanging the roles of $x_{1},x_{2}$, and using \eqref{eq: cont 1}, \eqref{eq: cont 2} and \eqref{eq: cont 7}, we get  the joint continuity of $\overline{H}$ in all its arguments.
\end{proof}
\begin{rem}\label{rem: lipschitz}
Note that \eqref{eq: cont 7} yields Lipschitz continuity of $x\mapsto \overline H(t,x,p, Y)$  provided $H$  is Lipschitz in $(x,y)$; the Lipschitz continuity in $y$ being needed in \eqref{eq: cont 4}. This observation will be useful in Step 5 of the proof of the main result.
\end{rem}

\subsection{The main result} We are now ready to state and prove our main convergence result. The last assumption we need is the following

\textit{Assumption \textbf{(D)}}
\begin{enumerate}
    \item[\textbf{(D)}] The matrix $\Sigma=\sigma\sigma^{\top}(x,y,u)$ has bounded second derivatives in $x$, uniformly in $(y,u)$ and at least one of the two conditions is satisfied:
    \begin{enumerate}
        \item  $\Sigma$ is independent of $y$ and $u$, i.e. $\sigma=\sigma(x)$;
	\item 	the drift of the fast process is independent of $x$, i.e. $b=b(y)$. 
    \end{enumerate}
\end{enumerate}
Assumption (D) ensures that the square root of $\Sigma$ is Lipschitz in $x$ (see \cite[Theorem 5.2.3, p.132]{stroock1997multidimensional}) and will be needed in Step 5 of the proof of our next result. For our main motivation as described in the introduction, assumption (D.a) is satisfied because $\sigma=0$. Assumption (D.b) on the other hand is relevant for applications in finance, see \cite{bardi2010convergence, fouque2011multiscale} and the references therein.

\begin{thm}
\label{thm-conv-value}
Under assumptions (A), (B), (C) and (D), the solution $V^{\varepsilon}$ to \eqref{HJB} converges uniformly on compact subsets of $(0,T)\times\mathds{R}^{n}\times\mathds{R}^{m}$ to the unique continuous viscosity solution of the limit problem \eqref{CP-limit HJB} satisfying a quadratic growth condition in $x$, i.e.
\begin{equation}
\label{quadratic gro in thm}
    \exists\;K>0\;\text{such that }\; |V(t,x)|\leq K(1+|x|^{2}),\quad\forall\;(t,x)\in[0,T]\times\mathds{R}^{n}
\end{equation}
\end{thm}  

\begin{rem}
Without assumption (D) we prove that the weak semilimits $\underline{V}$ and $\overline{V}$ are independent of $y$ and are, respectively, a super- and a subsolution of \eqref{CP-limit HJB}, as in \cite[Theorem 1]{alvarez2003singular}. If, in addition, we assume (D), we prove that the Comparison Principle holds for \eqref{CP-limit HJB}, which implies the uniform convergence of $V^{\varepsilon}$.
\end{rem}

The last auxiliary result we need is a Liouville property for semi-solutions of the PDE
\begin{equation}
\label{LV=0}
    -\mathcal{L}V(y) = -b(x,y)\cdot \nabla V(y) - \text{trace}(\varrho\varrho^{\top}(x,y) D^{2}V(y)) = 0,\quad \text{in }\mathds{R}^{m} ,
\end{equation}
where $x\in\mathds{R}^{n}$ is frozen, taken from  \cite[Theorem 2.1 \& 2.2]{bardi2016liouville} or \cite[Proposition 3.1]{mannucci2016ergodic}.
\begin{lem}
\label{liouville}
Assume there exist a function  $\omega\in C^{\infty}(\mathds{R}^{m})$ and $R_{0}>0$ such that
\begin{equation}
\label{lyapunov}
    -\mathcal{L}\omega \geq 0\quad \text{in } \overline{B(0,R_{0})}^{C},\quad \omega(y)\rightarrow +\infty \; \text{ as }\; |y|\rightarrow +\infty.
\end{equation}
Then every viscosity subsolution $V\in USC(\mathds{R}^m)$ to \eqref{LV=0} such that $\limsup\limits_{|y|\rightarrow\infty}\frac{V}{\omega}\leq 0$ and every viscosity supersolution $U\in LSC(\mathds{R}^m)$ to \eqref{LV=0} such that $\liminf\limits_{|y|\rightarrow\infty}\frac{U}{\omega}\geq 0$ are constant.
\end{lem}

\begin{proof}(Theorem \ref{thm-conv-value}) 
The proof follows the one of \cite[Theorem 5.1]{bardi2010convergence} (see also \cite[Theorem 3.2]{bardi2011optimal}). 

{Step 1.} We define the half-relaxed semilimits 
\begin{equation*}
    \underline{V}(t,x,y)= \liminf\limits_{\substack{\varepsilon\to 0\\t'\to t, x'\to x, y'\to y}} V^{\varepsilon}(t',x',y'),\quad \overline{V}(t,x,y)= \limsup\limits_{\substack{\varepsilon\to 0\\t'\to t, x'\to x, y'\to y}} V^{\varepsilon}(t',x',y')
\end{equation*}
for $t<T, x\in\mathds{R}^{n},y\in\mathds{R}^{m}$, and
\begin{equation*}
    \underline{V}(T,x,y)= \liminf\limits_{\substack{\varepsilon\to 0\\t'\to T^{-}, x'\to x, y'\to y}} V^{\varepsilon}(t',x',y'),\;\overline{V}(T,x,y)= \limsup\limits_{\substack{\varepsilon\to 0\\t'\to T^{-}, x'\to x, y'\to y}} V^{\varepsilon}(t',x',y').
\end{equation*}
By  \eqref{quadratic gro-V-eps} they also have quadratic growth, that is,
\begin{equation}
    |\underline{V}(t,x,y)|,\; |\overline{V}(t,x,y)|\; \leq K(1+|x|^2 + |y|^2),\quad \forall\; t\in [0,T],\; x\in \mathds{R}^{n},\: y\in \mathds{R}^{m}.
\end{equation}

{Step 2.} \textit{(We show that $\underline{V}(t,x,y),\overline{V}(t,x,y)$ do not depend on $y$ for every $t\in [0,T)$ and $x\in\mathds{R}^{n}$.)} Arguing as in Step 2 of the proof of \cite[Theorem 5.1]{bardi2010convergence}, we get that $\overline{V}(t,x,y)$ (resp., $\underline{V}(t,x,y)$) is, for every $t\in(0,T)$ and $x\in\mathds{R}^{n}$, a viscosity subsolution (resp., supersolution) to
\begin{equation}
    -\mathcal{L}(x,y,D_{y}V,D^{2}_{yy}V) = 0 \quad \text{ in } \mathds{R}^{m}
\end{equation}
where $\mathcal{L}$ is the differential operator defined in \eqref{inf gen}. Consider now the function $\omega$ defined on $\mathds{R}^m\setminus\{0\}$ such that
\begin{equation}
\label{omega-lypaunov}
    \omega(y) = \frac{1}{2}|y|^{2}\log|y|
\end{equation}
and such that $\omega(0)=0$.
It is easy to check that
\begin{equation*}
    \begin{aligned}
        \nabla \omega(y)  = \left(\frac{1}{2} + \log(|y|)\right)y \quad \text{ and } \quad 
        D^{2}\omega(y)  = \left(\frac{1}{2} + \log(|y|)\right)\mathds{I}_{m} + \frac{y\otimes y }{|y|^{2}} .
    \end{aligned}
\end{equation*}
Therefore, recalling $a=\varrho\varrho^{\top}$, one has
\begin{equation}
\label{eq: gen _ lyapunov}
\begin{aligned}
    -\mathcal{L}\omega  &=  -\left(\frac{1}{2} + \log(|y|)\right) (b(y)\cdot y) - \left(\frac{1}{2} + \log(|y|)\right)\text{trace}(a(y)) -
    \frac{1}{|y|^{2}}\text{trace}((y\otimes y)a(y)) \\
    & \geq -\left(\frac{1}{2} + \log(|y|)\right)\left((b(y)\cdot y) + m\overline{\Lambda}\,\right) -\overline{\Lambda}\, \xrightarrow[|y|\rightarrow \infty]{} +\infty
\end{aligned}
\end{equation}
thanks to assumption \eqref{recurrence condition} and \eqref{non degen}. Then  one can find $R>0$ such that
\begin{equation}
    -\mathcal{L}\omega(y) \geq 0 \quad \text{ in } \; \overline{B(0,R)}^{C},\quad \text{ and } \;\omega(y)\xrightarrow[|y|\rightarrow \infty]{}+\infty.
\end{equation}
We can now use Lemma \ref{liouville} with such a Lyapunov function $\omega$, since $\overline{V},\underline{V}$ have at most a quadratic growth in $y$, to conclude that the functions $y\mapsto \overline{V}(t,x,y)$, $ y\mapsto\underline{V}(t,x,y)$ are constants for every $(t,x)\in(0,T)\times\mathds{R}^{n}$. Finally, using the definition it is immediate to see that this implies that also $\overline{V}(T,x,y)$ and $\underline{V}(T,x,y)$ do not depend on $y$. 

{Step 3.}  \textit{(We show that $\overline{V}$ and $\underline{V}$ are sub and supersolutions to the PDE in \eqref{CP-limit HJB} in $(0,T)\times\mathds{R}^{n}$.)}   
The proof adapts the perturbed test function method \cite{evans1989perturbed, alvarez2003singular}. To show that $\overline{V}$ is a viscosity subsolution we fix $(\overline{t},\overline{x})\in (0,T)\times\mathds{R}^{n}$ and a smooth function $\psi$ such that $\psi(\overline{t},\overline{x})=\overline{V}(\overline{t},\overline{x})$ and $\overline{V}-\psi$ has a strict maximum at $(\overline{t},\overline{x})$. We must prove that
\begin{equation*}
\label{subsol}
    -\psi_{t}(\overline{t},\overline{x}) +\overline{H}(\overline{t},\overline{x},D_{x}\psi(\overline{t},\overline{x}),D^{2}_{xx}\psi(\overline{t},\overline{x})) + \lambda \overline{V}(\overline{t},\overline{x})\leq 0
\end{equation*}
Set $\overline{p}=D_{x}\psi(\overline{t},\overline{x})$,  $\overline{P}=D^{2}_{xx}\psi(\overline{t},\overline{x})$,  and 
assume by contradiction that for some $\eta>0$
\begin{equation*}
- \psi_{t}(\overline{t},\overline{x})  + \overline{H}(\overline{t},\overline{x}, \ov p, 
    \ov P ) + \lambda \psi(\ov{t},\ov{x})\geq 5\eta .
\end{equation*}
By the continuity of $\ov H$  {given by Proposition \ref{prop: cont}}, we can  choose $r>0$ such that
\begin{equation}
\label{contrad} 
- \psi_{t}(t,x)  + \overline{H}({t},{x},  \ov p, \ov P)  + \lambda
     \psi({t},{x} )\geq 4\eta 
\end{equation}
for all $(x,t)\in B((\overline{t},\overline{x}),r)$, and $\ep_o>0$ such that
\begin{equation}
\label{Heps}
|H^\ep({t},{x}, y, D_{x}\psi({t},{x}), D^{2}_{xx}\psi({t},{x}),0) - H(\overline{t},\overline{x}, y, \ov p, \ov P ,0)|< \eta
\end{equation}
for all $(x,t)\in B((\overline{t},\overline{x}),r)$, $y\in \ov B(0, R)$ ($R$ to be chosen soon),  and $\ep\leq\ep_o$.
Now consider, as in \eqref{dirichlet-Dn},  the $\delta(n)$-cell problem
\begin{equation}
\label{cell-prob}
\left\{
    \begin{aligned}
        \delta \chi_{\delta}(y) -\mathcal{L}(\overline{x},y,D\chi_{\delta},D^{2}\chi_{\delta}) + H(\overline{t},\overline{x},y,\overline{p},\overline{P},0) &=& 0,&\quad \text{ in }\;D_{n} ,\\
        \chi_{\delta}(y) &=& 0,&\quad \text{ in }\;  
        \partial D_{n},
    \end{aligned}
\right.
\end{equation}
where $\delta:=\delta(n)=O\left(\frac{1}{n^{4+\alpha}}\right)$ and $D_{n}=B(0,n)$. By Proposition \ref{eff ham} there exists $n_{o}>0$  such that, for every $n\geq n_{o}$,  $R<n_o$,
\begin{equation}
\label{approx-corr}
    |\delta \chi_{\delta}(y) + \overline{H}(\overline{t},\overline{x},\overline{p},\overline{P})| \leq \eta 
    , \quad \forall y\in B(0,{R}) .
\end{equation}
Moreover 
\begin{equation}
\label{approxL}
|\mathcal{L}(\overline{x},y,D\chi_{\delta},D^{2}\chi_{\delta}) - \mathcal{L}({x},y,D\chi_{\delta},D^{2}\chi_{\delta})|<\eta
\end{equation}
for $|x-\ov x|<r$, by decreasing $r$ if necessary, and we set $C_n :=\max_{\ov B(0,{R})} |\chi_{\delta}(y) |$.  We define the perturbed test function
\begin{equation}
    \psi^{\varepsilon}(t,x,y) := \psi(t,x) + \varepsilon\chi_{\delta}(y) ,
\end{equation}
which is in $C^2(\ov \Omega)$ for $\Omega:= B((\overline{t},\overline{x}),r)\times B(0,{R})$. We claim that $\psi^\ep$ is a strict supersolution of the PDE \eqref{HJB} in $\Omega$  for $\ep\leq\ep _o$ and $\ep \la C_n<\eta$. In fact
\begin{equation}\label{strict}
\begin{aligned}
& -\psi^\ep_t(t,x) + H^\ep({t},{x}, y, D_{x}\psi({t},{x}), D^{2}_{xx}\psi({t},{x}),0)  - \mathcal{L}(x,y, D\chi_{\delta},D^{2}\chi_{\delta}) + \la \psi^\ep_t(t,x)  \\
&\quad \quad \quad  \geq - \psi_t(t,x) + H^\ep({t},{x}, y, D_{x}\psi({t},{x}), D^{2}_{xx}\psi({t},{x}),0) -  \delta \chi_{\delta}(y) - H(\overline{t},\overline{x},y,\overline{p},\overline{P},0)+ \la \psi^\ep_t(t,x) \\
&\quad \quad \quad  \geq -\psi_t(t,x) -  \eta +  \overline{H}(\overline{t},\overline{x},\overline{p},\overline{P}) - \eta + \la \psi_t(t,x)  +\la\ep \chi_{\delta}(y)\\
&\quad \quad \quad  \geq 4\eta - 2\eta - \la\ep C_n \geq \eta  >0
\end{aligned}
\end{equation}
where in the first inequality we used \eqref{approxL} and \eqref{cell-prob},  in the  second inequality we used \eqref{Heps} and \eqref{approx-corr}, and in the third inequality we used \eqref{contrad}.

Since the maximum of $\overline{V}-\psi$ at $(\overline{t},\overline{x})$ is strict, we can decrease $r$ so that $\overline{V}-\psi\leq-2\eta$ on $\partial \Omega$.
Moreover 
\begin{equation}
\label{ep_lim}
    \limsup\limits_{\substack{\varepsilon\to 0\\t'\to t, x'\to x, y'\to y}} V^{\varepsilon}(t',x',y') - \psi^{\varepsilon}(t',x',y') = \overline{V}(t,x) - \psi(t,x) 
\end{equation}
and the compactness of $\partial \Omega$ imply that $V^{\varepsilon} - \psi^{\varepsilon}\leq-\eta$ on $\partial \Omega$ for $\ep$ small enough. We claim that, for such $\ep$,
\begin{equation}
\label{wrong}
V^{\varepsilon} - \psi^{\varepsilon}\leq-\eta \qquad\text{in } \Omega.
\end{equation}
In fact, if this is not the case, $V^{\varepsilon} - \psi^{\varepsilon}$ has a maximum point in $\Omega$, a contradiction to the fact that $V^{\varepsilon}$ is a viscosity subsolution of \eqref{HJB} in $\Omega$ and $\psi^{\varepsilon}$ satisfies \eqref{strict}.  Now \eqref{ep_lim} and \eqref{wrong} imply $\overline{V}(\overline{t},\overline{x})<\psi(\overline{t},\overline{x})$, which is a contradiction and completes the proof that $\overline{V}$ is a subsolution to \eqref{CP-limit HJB}.
The proof that $\underline{V}$ is a supersolution is completely analogous.

{Step 4.} \textit{(Behavior of $\overline{V}$ and $\underline{V}$ at time T)} In this step, we adapt the \textit{Step 4} in the proof of \cite[Theorem 5.1]{bardi2010convergence} or in \cite[Theorem 3.2]{bardi2011optimal} using our result in Proposition \ref{prop: effective initial data}. The main difference relies in the use of the sequence of Cauchy problems with bounded domains \eqref{eq: cauchy prob_initial data - truncated} instead of the Cauchy problem \eqref{eq: cauchy prob_initial data} that was used in \cite{bardi2011optimal,bardi2010convergence}. We repeat the proof for the sake of consistency and clarity.\\
We prove only the statement for subsolution, since the proof for the supersolution is completely analogous.

We fix $\overline{x}\in \mathds{R}^n$ and $t_{0}>0$, and we consider, for some $n>0$ to be later made precise, the unique bounded solution $\omega^{r,n}$ to the Cauchy problem in $[0,T(n)]\times D_{n}$ where $T(n) \coloneqq n^{2}t_{0}$ and $D_{n}$ is the ball of radius $n$ in $\mathds{R}^m$
\begin{equation}
\label{eq: proof - conv val - cauchy prob - n}
\left\{
    \begin{aligned}
         \omega_{t} - \mathcal{L}(\overline{x},y,D\omega,D^{2}\omega) = 0,&\quad \text{in }\, (0,T(n)]\times D_{n},\\
         \omega (0,y) = \sup\limits_{\{|x-\overline{x}|\leq r\}} g(x,y),&\quad \text{in } D_{n},\\
         \omega(t,y) = 0, &\quad \text{in } [0,T(n)]\times \partial D_{n}.
    \end{aligned}
\right.
\end{equation}
Using stability properties of viscosity solutions it is not hard to see that $\omega^{r,n}$ converges, as $r\to 0$, to the solution $\omega^{n}$ of \eqref{eq: cauchy prob_initial data - truncated} set in $[0,T(n)]\times D_{n}$. We recall that
\begin{equation*}
    \overline{g}(\overline{x})\coloneqq\mu_{\overline{x}}(g(\overline{x},\cdot))=\int_{\mathds{R}^{m}}g(\overline{x},y)\,\text{d}\mu_{\overline{x}}(y).
\end{equation*}
Using the convergence result in Proposition \ref{prop: effective initial data} and the uniform convergence of $\omega^{r,n}$ to $\omega^{n}$, it is easy to see that for every $\eta>0$ there exist $r_{0}$ and $n_{0}>0$ such that
\begin{equation}
\label{eq: proof - conv omega - approximation}
    \forall\,n\geq n_{0}\,:\quad |\omega^{r,n}(T(n),y) - \overline{g}(\overline{x})|\leq \eta,\quad\forall\,r<r_{0},\,y\in D_{n}\supseteq \overline{D}_{n_{0}}.
\end{equation}
We now fix $r<r_{0}$ and a constant $M_{r}$ such that $V^{\varepsilon}(t,x,y)\leq M_{r}$ and $|g(x,y)|\leq M_{r}/2$ for every $\varepsilon>0$, $x\in \overline{B}:=\overline{B(\overline{x},r)}$ and $y\in \overline{D}:=\overline{D}_{n_{0}}$. 
This is possible  by Proposition \ref{prelimit sol} and assumption \eqref{assumption-cost}. 
Moreover we fix a smooth nonnegative function $\psi$ such that $\psi(\overline{x})=0$ and $\psi(x)+\inf_{y\in \overline{D}}g(x,y)\geq 2M_{r}$ for every $x\in\partial B$
(which is easy to build because  $\inf_{x\in \partial B}\inf_{y\in \overline{D}}g(x,y)\geq -M_{r}/2$). Let $C_{r}$ be a positive constant such that
\begin{equation*}
    |H^{\varepsilon}(t,x,y,D\psi(x),D^{2}\psi(x),0)|\leq C_{r}\quad \text{for } x\in \overline{B},\;  y\in \overline{D} \,\text{ and }\, \varepsilon>0
\end{equation*}
where $H^{\varepsilon}$ is defined in \eqref{H-eps}. Note that such a constant exists thanks to assumptions \eqref{assumption-slow} and \eqref{assumption-cost}. We define the function 
\begin{equation*}
    \psi^{\varepsilon}_{r}(t,x,y) = \omega^{r,n}\left(\frac{T-t}{\varepsilon},y\right) + \psi(x) + C_{r}(T-t),
\end{equation*}
for some fixed $n>n_{0}$, and we claim that it is a supersolution to the parabolic problem
\begin{equation}
\label{eq: parabolic problem - proof - final time}
\left\{
    \begin{aligned}
         -V_{t} + F^{\varepsilon}\left(t,x,y,V,D_{x}V,\frac{D_{y}V}{\varepsilon}, D^{2}_{xx}V^{\varepsilon},\frac{D^{2}_{yy}V}{\varepsilon},\frac{D^{2}_{xy}V}{\sqrt{\varepsilon}}\right) &= 0, \text{ in } (0,T)\times  B \times \overline{D}\\
         V(t,x,y) &= M_{r}, \text{ in } (0,T)\times\partial B\times \overline{D}\\
         V(T,x,y) &= g(x,y), \text{ in } \overline{B}\times\overline{D}
    \end{aligned}
\right.
\end{equation}
where $F^{\varepsilon}$ is defined in \eqref{Ham F}. Indeed
\begin{equation*}
\begin{aligned}
    &-(\psi^{\varepsilon}_{r})_{t} + F^{\varepsilon}\left(t,x,y,D_{x}\psi^{\varepsilon}_{r},\frac{D_{y}\psi^{\varepsilon}_{r}}{\varepsilon},D^{2}_{xx}\psi^{\varepsilon}_{r}, \frac{D^{2}_{yy}\psi^{\varepsilon}_{r}}{\varepsilon},\frac{D^{2}_{xy}\psi^{\varepsilon}_{r}}{\sqrt{\varepsilon}}\right) \\
    &\quad \quad \quad \quad = \frac{1}{\varepsilon}\left[ (\omega^{r,n})_{t}-\mathcal{L}(y,D\omega^{r,n},D^{2}\omega^{r,n}) \right] + C_{r} + H^{\varepsilon}(t,x,y,D\psi(x),D^{2}\psi(x),0) \, \geq 0.
\end{aligned}
\end{equation*}
Moreover $\psi^{\varepsilon}_{r}(T,x,y) = \sup\limits_{\{|x-\overline{x}|\leq r\}}g(x,y) + \psi(x)\geq g(x,y)$. \\
Finally, observe that the constant function $\min\{\,0\,;\,\inf\limits_{y\in \overline{D}}\sup\limits_{\{|x-\overline{x}|\leq r\}} g(x,y)\,\}$
is always a subsolution to \eqref{eq: proof - conv val - cauchy prob - n} and then by a standard comparison principle we obtain 
$$
\omega^{r,n}(t,y)\geq \min\{\,0\,;\,\inf\limits_{y\in \overline{D}}\sup\limits_{\{|x-\overline{x}|\leq r\}} g(x,y)\,\}.
$$
 This implies, for all $x\in\partial B$,
\begin{equation*}
\begin{aligned}
    \psi^{\varepsilon}_{r}(t,x,y) & \geq \, \min\{\,0\,;\,\inf\limits_{y\in \overline{D}}\sup\limits_{\{|x-\overline{x}|\leq r\}} g(x,y)\,\}+ 2M_{r} - \inf\limits_{y\in\overline{D}}g(x,y) + C_{r}(T-t) \geq \, M_{r}
\end{aligned}
\end{equation*}
where we have used either the fact that $|g(x,y)|\leq M_{r}/2$, and hence $- \inf\limits_{y\in\overline{D}}g(x,y)\geq -M_{r}/2$, when we have  $\min\{\,0\,;\,\inf\limits_{y\in \overline{D}}\sup\limits_{\{|x-\overline{x}|\leq r\}} g(x,y)\,\}= 0$, or otherwise, we have used the fact that $\inf\limits_{y\in \overline{D}}\sup\limits_{\{|x-\overline{x}|\leq r\}} g(x,y)\, - \inf\limits_{y\in\overline{D}}g(x,y) \geq 0$. In the first case, we get $\psi^{\varepsilon}_{r}(t,x,y)\geq 3M_{r}/2$ and in the second case we have $\psi^{\varepsilon}_{r}(t,x,y)\geq 2M_{r}$. Then $\psi^{\varepsilon}_{r}$ is a supersolution to \eqref{eq: parabolic problem - proof - final time}. For our choice of $M_{r}$ we get that $V^{\varepsilon}$ is a subsolution to \eqref{eq: parabolic problem - proof - final time}. Moreover both $V^{\varepsilon}$ and $\psi^{\varepsilon}_{r}$ are bounded in $[0,T]\times \overline{B}\times \overline{D}$, because of the estimate \eqref{quadratic gro-V-eps}, of the boundedness of $\omega^{r,n}$ and of the regularity of $\psi$. So, a standard comparison principle for viscosity solutions gives
\begin{equation*}
\begin{aligned}
    V^{\varepsilon}(t,x,y)  \leq \, \psi^{\varepsilon}_{r}(t,x,y) = \omega^{r,n}\left(\frac{T-t}{\varepsilon},y\right) + \psi(x) + C_{r}(T-t)
\end{aligned}
\end{equation*}
for every $0<r<r_{0}$, $n>n_{0}$ $\varepsilon>0$, $(t,x,y)\in [0,T]\times \overline{B}\times \overline{D}$. We compute the upper limit of both sides of the previous inequality as $(\varepsilon,t,x,y)\to (0,t',x',y')$ for $t'\in (0,T)$, $x'\in B$, $y'\in D$ and $\varepsilon\coloneqq\varepsilon(n)=\frac{T-t}{T(n)}$ (recalling $T(n)=n^{2}t_{0})$ and get, using \eqref{eq: proof - conv omega - approximation},
\begin{equation*}
    \overline{V}(t',x')\leq \overline{g}(\overline{x}) + \eta + \psi(x') + C_{r}(T-t').
\end{equation*}
Then taking the upper limit for $(t',x')\to (T,\overline{x})$, we obtain obtain $\overline{V}(T,\overline{x})\leq \overline{g}(\overline{x}) + \eta$ which permits us to conclude recalling that $\eta$ is arbitrary.

The proof for $\underline{V}$ is completely analogous, once we replace the Cauchy problem \eqref{eq: proof - conv val - cauchy prob - n} with 
\begin{equation*}
\left\{
    \begin{aligned}
         \omega_{t} - \mathcal{L}(y,D\omega,D^{2}\omega) = 0, &\quad \text{in }\, (0,T(n)]\times D_{n},\\
         \omega (0,y) = \inf\limits_{\{|x-\overline{x}|\leq r\}} g(x,y), &\quad \text{in } D_{n},\\
         \omega(t,y) = 0, &\quad \text{in } [0,T(n)]\times \partial D_{n}.
    \end{aligned}
\right.
\end{equation*}

\hfill\\
{Step 5.} \textit{(Uniform convergence)}.
We observe that by definition $\overline{V}\geq \underline{V}$ and that both $\overline{V}$ and $\underline{V}$ satisfy the same quadratic growth condition \eqref{quadratic gro in thm}.  Moreover the Hamiltonian $\overline{H}$ defined in \eqref{eff-Hamiltonian} can be written thanks to Proposition \ref{representation-prop} as a Bellman Hamiltonian of the form
\begin{equation*}
    \overline{H}(t,x,p,P)=\min\limits_{\nu\in L^{\infty}(\mathds{R}^{m},U)}\bigg\{\, -\text{trace}(\overline{\sigma}\,\overline{\sigma}^{\top}P) -  \overline{f}\cdot p  - \overline{\ell} \; \bigg\}
\end{equation*}
where 
\begin{equation*}
\begin{aligned}
    &\overline{\sigma}  =\overline{\sigma}(x,\nu)=\sqrt{\int_{\mathds{R}^{m}}\sigma\sigma^{\top}(x,y,\nu(y))\,\text{d}\mu_{x}(y)}\,\\
    & \overline{f}  = \overline{f}(x,\nu)=\int_{\mathds{R}^{m}}f(x,y,\nu(y))\,\text{d}\mu_{x}(y)\\
    & \overline{\ell}  = \overline{\ell}(t,x,\nu)=\int_{\mathds{R}^{m}}\ell(t,x,y,\nu(y))\,\text{d}\mu_{x}(y).
\end{aligned}
\end{equation*}
Under assumptions (D), we actually have in the case (D.a)
\begin{equation*}
    \overline{\sigma} = \sqrt{\sigma\sigma^{\top}(x)},
\end{equation*}
and in the case (D.b), the invariant measure of the fast process $Y$ does not depend on $x$. 
Therefore, $\overline{\sigma},\overline{f},\overline{\ell}$ inherit regularity and growth conditions of $\sigma,f,\ell$ thanks to assumptions (A), (B), (C), (D) and Remark \ref{rem: lipschitz}, and  fall in the framework of \cite{da2006uniqueness}. Hence we can use the comparison result between sub- and supersolutions to parabolic problems satisfying a quadratic growth condition, given in \cite[Theorem 2.1]{da2006uniqueness}, to deduce $\underline{V}\geq \overline{V}$. Therefore $\underline{V}=\overline{V}=:V$. In particular $V$ is continuous, and by definition of half-relaxed semilimits, this implies that $V^{\varepsilon}$ converges locally uniformly to $V$ (see \cite[Lemma V.1.9]{bardi2008optimal}).
\end{proof}

\subsection*{Acknowledgements} The authors wish to thank Markus Fischer for useful conversations and for pointing out \cite[Prop. 1.4]{herrmann2006transition} used in the proof of Lemma \ref{conv-exit time}. The authors also thank the two anonymous referees for the many useful comments.

\bibliographystyle{amsplain}
\bibliography{bibliography}

\end{document}